\documentclass[11pt,oneside,english]{amsart}
\usepackage[T1]{fontenc}
\usepackage[latin9]{inputenc}
\usepackage{geometry}
\usepackage{amstext}
\usepackage{amsthm}
\usepackage{amssymb}
\usepackage{url}
\usepackage{hyperref}
\usepackage{graphicx}
\usepackage{xcolor}
\usepackage{grffile}

\newcommand{\zblue}[1]{#1}

\newcommand{\cF}{\mathcal{F}}
\newcommand{\cG}{\mathcal{G}}
\newcommand{\tbx}{\tilde\bx}
\newcommand{\del}[1]{}
\makeatletter
\theoremstyle{plain}
\newtheorem{thm}{\protect\theoremname}
\theoremstyle{definition}
\newtheorem{defn}[thm]{\protect\definitionname}
\theoremstyle{plain}
\newtheorem{lem}[thm]{\protect\lemmaname}
\theoremstyle{plain}
\newtheorem{prop}[thm]{\protect\proppositionname}
\theoremstyle{plain}
\newtheorem{cor}[thm]{\protect\corollaryname}
\theoremstyle{plain}
\newtheorem{conj}[thm]{\protect\conjecturename}
\theoremstyle{definition}
\newtheorem*{problem*}{\protect\problemname}
\theoremstyle{remark}
\newtheorem{rem}[thm]{\protect\remarkname}

\makeatother

\usepackage{babel}
\providecommand{\definitionname}{Definition}
\providecommand{\lemmaname}{Lemma}
\providecommand{\proppositionname}{Proposition}
\providecommand{\problemname}{Problem}
\providecommand{\theoremname}{Theorem}
\providecommand{\remarkname}{Remark}
\providecommand{\conjecturename}{Conjecture}
\providecommand{\corollaryname}{Corollary}

\evensidemargin \oddsidemargin
\begin{document}
\global\long\def\be{\mathbf{e}}%
\global\long\def\bf{\mathbf{f}}%
\global\long\def\bx{\mathbf{x}}%
\global\long\def\by{\mathbf{y}}%
\global\long\def\ba{\mathbf{a}}%
\global\long\def\bq{\mathbf{q}}%
\global\long\def\bp{\mathbf{p}}%
\global\long\def\bb{\mathbf{b}}%
\global\long\def\bv{\mathbf{v}}%
\global\long\def\bq{\mathbf{q}}%
\global\long\def\bs{\mathbf{s}}%
\global\long\def\bw{\mathbf{w}}%
\global\long\def\br{\mathbf{r}}%
\global\long\def\bz{\mathbf{z}}%
\global\long\def\bad{\mathbf{Bad}}%
\global\long\def\badw{\mathbf{Bad}\left(\bw\right)}%
\global\long\def\cala{\mathcal{A}}%
\global\long\def\calu{\mathcal{U}}%
\global\long\def\calh{\mathcal{H}}%
\global\long\def\calb{\mathcal{B}}%
\global\long\def\calq{\mathcal{Q}}%
\global\long\def\calc{\mathcal{C}}%
\global\long\def\cale{\mathcal{E}}%
\global\long\def\calf{\mathcal{F}}%
\global\long\def\calw{\mathcal{W}}%
\global\long\def\bbr{\mathbb{R}}%
\global\long\def\bbz{\mathbb{Z}}%
\global\long\def\bbq{\mathbb{Q}}%
\global\long\def\bbn{\mathbb{N}}%
\global\long\def\sl{\text{SL}}%
\global\long\def\sltr{\text{SL}_{3}\left(\bbr\right)}%
\global\long\def\sltz{\text{SL}_{3}\left(\bbz\right)}%
\global\long\def\sldr{\text{SL}_{d+1}\left(\bbr\right)}%
\global\long\def\sldz{\text{SL}_{d+1}\left(\bbz\right)}%
\global\long\def\sep{\,:\,}%
\global\long\def\subp{\supseteq}%
\global\long\def\subq{\subseteq}%
\global\long\def\eps{\varepsilon}%
\global\long\def\diam{\operatorname{diam}}%
\global\long\def\diag{\operatorname{diag}}%
\global\long\def\dist{\operatorname{d}}%
\global\long\def\linearspan{\operatorname{span}}%
\global\long\def\support{\operatorname{supp}}%
\global\long\def\supp{\operatorname{supp}}%
\global\long\def\covolume{\operatorname{covolume}}%

\global\long\def\limx#1#2{{\displaystyle \lim_{#1\to\infty}}#2}%

\global\long\def\xto#1#2{{\displaystyle \xrightarrow[#1\to#2]{}}}%
\newcommand{\R}{\mathbb{R}}
\newcommand{\N}{\mathbb{N}}
\newcommand{\cC}{\mathcal{C}}
\newcommand{\Z}{\mathbb{Z}}
\newcommand{\ve}{\varepsilon}
\newcommand{\vv}[1]{\mathbf{#1}}

\newcommand{\three}{1}
\newcommand{\threeo}{}
\newcommand{\threePlusOne}{2}

\global\long\def\modified{\text{restricted}}%

\title{$\badw$ is hyperplane absolute winning}

\author[Victor Beresnevich]{Victor Beresnevich}
\address{Victor Beresnevich, Department of Mathematics, University of York, Heslington, York, YO10 5DD, United Kingdom}
\email{victor.beresnevich@york.ac.uk}

\author[Erez Nesharim]{Erez Nesharim}
\address{Erez Nesharim, Einstein Institute of Mathematics, The Hebrew University of Jerusalem, Jerusalem, 9190401, Israel}
\email{ereznesh@gmail.com}

\author[Lei Yang]{Lei Yang\vspace{-2.85em}}
\address{Lei Yang, College of Mathematics, Sichuan University, Chengdu, Sichuan, 610000, China}
\email{lyang861028@gmail.com}

\date{}

\maketitle

\begin{abstract}
In 1998 Kleinbock conjectured that any set of weighted badly approximable $d\times n$ real matrices is a winning subset in the sense of Schmidt's game. In this paper we prove this conjecture in full for vectors in $\R^d$ in arbitrary dimensions by showing that the corresponding set of weighted badly approximable vectors is hyperplane absolute winning. The proof uses the Cantor potential game played on the support of Ahlfors regular absolutely decaying measures and 
the quantitative nondivergence estimate for a class of fractal measures due to Kleinbock, Lindenstrauss and Weiss. To establish the existence of a relevant winning strategy in the Cantor potential game we introduce a new approach using two independent diagonal actions on the space of lattices.
%One of the actions arises from the Dani correspondence and is used to detect badly approximable vectors. The second action extends the orbits into expanding compact sets, ultimately enabling the use of the quantitative non-divergence estimate.
\end{abstract}

\paragraph{}
\begin{center}
\emph{Dedicated to Anna Nesharim\vspace{0.85em}}
\end{center}

\section{Introduction}\label{sec:introduction}

As is well known, the rational points are dense in the real space $\R^d$, meaning that $\R^d$ can be covered by cubes in $\R^d$ of an arbitrarily small fixed sidelength $\ve>0$ centred at rational points. Various quantitative aspects of this basic property are studied within the theory of Diophantine approximation. For instance, by Dirichlet's theorem, $\R^d$ can be covered by cubes in $\R^d$ of sidelength $2q^{-(d+1)/d}$ centred at rational points {(not necessarily written in the lowest terms)} with arbitrarily large denominators $q\in\bbn$. One of the fundamental concepts studied in Diophantine approximation is that of badly approximable points. These are precisely the points in $\R^d$ that {cannot be} covered by the cubes arising from Dirichlet's theorem when $2$ is replaced by any positive constant. In the more general case one considers coverings by parallelepipeds with different sidelengths controlled by $d$ real parameters referred to as weights. This more general setup gives rise to the notion of weighted badly approximable points that will be the main object of study in this paper.

In what follows $d\in\bbn$ and $\calw_d$ denotes the collection of all \emph{$d$-dimensional weights}:
\[
\calw_d=\left\{\bw=(w_1,\dots,w_d)\in\bbr^d\sep w_1,\ldots,w_d\geq0,\;w_1+\ldots+w_d=1\right\}.
\]
For $\bw\in\calw_d$, a vector $\bx{=(x_1,\dots,x_d)}\in\bbr^d$ is called \emph{badly approximable with respect to $\bw$} if there exists $c>0$ such that for every $q\in\bbn$ and $\bp=(p_1,\dots,p_d)\in\bbz^d$ there exists $1\leq i\leq d$ satisfying
\[
\left|x_i-\frac{p_i}{q}\right|\geq\frac{c}{q^{1+w_i}}\,.
\]
Let $\badw$ be the \emph{set of badly approximable vectors in $\R^d$ with respect to $\bw$}.
%Thus,
%\[
%\badw=\left\{  \bx\in\bbr^{d}\sep\inf_{q\in\bbn,\;\bp\in\bbz^d}\max_{1\leq i\leq d}q^{w_{i}}\left| qx_i-p_i \right|>0\right\}.
%\]

One of the motivations for studying the set of weighted badly approximable vectors comes from its
connection to a conjecture of Littlewood -- a famous open problem from the 1930s. Let us briefly recall this connection.

\begin{conj}[Littlewood's conjecture, 1930s]
\label{conj:lit}
Every $\bx{=(x_1,x_2)}\in\bbr^2$ satisfies
\begin{equation}\label{eq:lit}
\inf_{q\in\bbn,\;\bp\in\bbz^2}q\,|qx_1-p_1|\,|qx_2-p_2|=0\,.
\end{equation}
\end{conj}

It was noted by Schmidt \cite{Schmidt7} that if $\bx\notin \badw$ for some $\bw\in\calw_2$
then $\bx$ satisfies \eqref{eq:lit}. In particular, if the intersection of the sets $\badw$ over all $\vv w\in\calw_2$ was the empty set, then Littlewood's conjecture would follow. However, Schmidt doubted that using only two weights would be sufficient, if his observation can be used to verify \eqref{conj:lit} at all. Specifically, Schmidt formulated the following problem that has {inspired} many researchers in Diophantine approximation and homogeneous dynamics.% for at least three decades.

\begin{conj}[Schmidt's conjecture, 1982]
\label{conj:schmidt}
{For every} $\bw_1,\bw_2\in\calw_2$ {we have that}
\begin{equation}\label{eq:schmidt}
\bad(\bw_1)\cap\bad(\bw_2)\neq\varnothing\,.
\end{equation}
\end{conj}

%\noindent
Almost three decades later Schmidt's conjecture was verified by Badziahin, Pollington and Velani in the tour de force \cite{BPV}, which opened the way to many exciting new developments.

The more general version of Schmidt's conjecture deals with arbitrary finite and, furthermore, countable intersections of $\badw$. Already in \cite{BPV} arbitrary finite intersections were considered. In fact, {the main result of \cite{BPV} implies} that
\begin{equation}\label{eq:N13}
\bigcap_{n=1}^\infty\bad(\bw_n)\neq\varnothing
\end{equation}
if the countably many weights $\bw_1,\bw_2,\ldots\in\calw_2$ satisfy the condition {that}
\begin{equation}\label{eq:BPVcondition}
\liminf_{n\to\infty}\min\bw_n > 0\,.
\end{equation}
Using different techniques condition \eqref{eq:BPVcondition} was independently removed by An \cite{An2} and the second named author \cite{N13}, who both established \eqref{eq:N13} for arbitrary countable intersections. Indeed, An \cite{An2} showed a stronger dimension statement.

Schmidt's conjecture can {also} be considered in higher dimensions. In this generality it was verified by the first named author \cite{Beresnevich_BA}. Similarly to the two dimensional result of \cite{BPV}, {\eqref{eq:N13} was established in \cite{Beresnevich_BA}} for any sequence of weights {$\bw_1,\bw_2,\ldots\in\calw_d$} satisfying \eqref{eq:BPVcondition}. Condition \eqref{eq:BPVcondition} was finally removed by the third named author {in} \cite{Yang}. {In should be noted that these papers go the extra mile to give a full dimension statement for the intersection appearing in \eqref{eq:N13} and enable to restrict the left hand side of \eqref{eq:N13} to nondegenerate curves and manifolds.}

Two natural {frameworks} for proving the countable intersection property {of the sets $\badw$} are {offered by} topology and measure theory.
Indeed, if $X$ is a complete metric space or a measure space and $S_1,S_2,\ldots\subq X$ are $G_\delta$
dense, or, respectively, full measure sets, then $\bigcap_{n=1}^\infty S_n$ is $G_\delta$ dense,
respectively, a set of full measure, and in particular, nonempty. However, the set $\badw$ is neither comeagre nor conull. In fact, $\badw$ is a countable union of closed sets whose Lebesgue measure is zero, hence it is both meagre and null.

An alternative framework to establish the countable intersection property is {offered by} game theory. This was first
articulated by Schmidt \cite{schmidt1} who introduced a variant of the Banach-Mazur game, now called
Schmidt's game, and its corresponding winning sets. Ever since other variants of Schmidt's game have been {proposed} by many authors for {various} purposes. We refer the reader to Section~\ref{sec:schmidt} for the definitions of \emph{hyperplane absolute winning sets \emph{(abbr. \emph{HAW})}} and \emph{Cantor winning sets}, which will be used in this paper. 
We also refer the reader to \cite[\S2]{BHNS} or \cite[\S2]{BFKRW} for the definitions of \emph{winning sets}, \emph{$\alpha$-winning sets} and \emph{absolute winning sets} which will be mentioned in this paper.

The study of winning properties of $\badw$ has a long history.
Schmidt proved in \cite{schmidt1} that $\bad(1)$ is winning, where it was also
mentioned that the analogous theorem holds for
$$\bw = \bw_d:=\left(\tfrac{1}{d},\ldots,\tfrac{1}{d}\right)$$
for every $d$. Indeed, the proof of this can be found in Schmidt's monograph \cite{Schmidt3}. McMullen \cite{McMullen_absolute_winning} proved that $\bad(1)$ is absolute winning.
Later Broderick, Fishman, Kleinbock, Reich and Weiss \cite{BFKRW} proved that $\bad(\bw_d)$ is
HAW for any $d\geq1$.

{However, the study of weighted badly approximable points turned out to be much harder. Indeed, the following natural problem that was raised by Kleinbock \cite[Section 8]{Kleinbock_matrices} over two decades ago remains open with the exception of one special case that will shortly be mentioned.}

\begin{problem*}[Kleinbock, 1998]
    Is it true that $\badw$ is winning for every weight $\bw$?
\end{problem*}

The {first breakthrough came about with the paper of} An \cite{An1} who {settled it for $d=2$.} Based on \cite{An1}, Simmons and the second named author \cite{NesharimSimmons} proved that $\badw$ is HAW for any $\bw \in \calw_2$. In higher dimensions, the only known result towards Kleinbock's problem is due to Guan and Yu \cite{guanyu} who proved that for weights $\bw \in \calw_d$ satisfying {the condition}
$$
w_1=\dots=w_{d-1}\ge w_d\,,
$$
$\badw$ is HAW.
% In fact,
% in its equivalent formulation Theorem~\ref{thm:intersection} was already proved independently by Kleinbock and Weiss \cite{KleinbockWeiss1} and by Kristensen, Thorn and Velani \cite{KTV}.
The goal of this paper is to resolve Kleinbock's problem in full. Our main result reads as follows.

\begin{thm}\label{thm:HAW}
For any $\bw\in\calw_d$ the set $\badw$ is HAW. In particular, it is winning.
\end{thm}

The HAW property implies more than just the countable intersection property. For example, we have the following corollary, which follows from Theorem~\ref{thm:HAW} on applying properties of HAW sets established in \cite{BFKRW} (see Section~\ref{sec:schmidt} for the definition of \emph{Ahlfors regular} and \emph{absolutely decaying} measures).

\begin{cor}\label{cor:HAW}
For any sequence of weights $\bw_1,\bw_2\ldots\in\calw_d$ and any sequence $f_1,f_2,\ldots$ of $C^1$ diffeomorphisms of $\bbr^d$, the set
\[
\bigcap_{n=1}^\infty f_n\left(\bad\left(\bw_n\right)\right)
\]
is HAW. In particular, for every Ahlfors regular absolutely decaying measure $\mu$ on $\bbr^d$ we have that
\[
\dim\left(\bigcap_{n=1}^\infty f_n\left(\bad\left(\bw_n\right)\right)\cap \support\mu\right)=\dim{\left(\supp\mu\right)},
\]
where $\dim$ stands for Hausdorff dimension.
\end{cor}

Theorem~\ref{thm:HAW} is proved by passing to the following equivalent formulation.

\begin{thm}\label{thm:intersection}
For any $\bw\in\calw_d$ and any compactly supported Ahlfors regular absolutely decaying measure $\mu$ on $\bbr^d$ we have that
\begin{equation}\label{eq:intersection}
\badw\cap \support \mu\neq \varnothing\,.
\end{equation}
\end{thm}

Over the last two decades Schmidt's conjecture motivated significant amount of research concerning badly approximable points in fractals, starting with Pollington and Velani \cite{PollingtonVelani02} and Kleinbock and Weiss \cite{KleinbockWeiss1}.
Initial progress towards Theorem~\ref{thm:intersection} was made in \cite{KleinbockWeiss1} for $\bw=\bw_d$
and in \cite{KTV}, where \eqref{eq:intersection} was proved for product measures $\mu=\mu_1\times\cdots\times\mu_d$ with each $\mu_i$ being Ahlfors regular. Other notable developments include those by Fishman \cite{fishman} and Kleinbock and Weiss \cite{KleinbockWeiss2}.

The tools used in the proof of Theorems~\ref{thm:HAW} and \ref{thm:intersection} are the Cantor potential game which was introduced by Badziahin, Harrap, Simmons and the second named author \cite{BHNS}, and the quantitative nondivergence estimate for ``friendly'' measures due to Kleinbock, Lindestrauss and Weiss \cite{KLW}, albeit, within this paper, the latter is only applied in the context of Ahlfors regular absolutely decaying measures.

In order to shed some light on the new ideas involved in the proof of Theorem~\ref{thm:HAW},
it is useful to compare the results in this paper to those of \cite{BNY1} and several preceding publications, which deal with badly approximable points on nondegenerate curves in $\R^d$. For simplicity we restrict our discussion to analytic nondegenerate curves. Let $\bf:I_0\to\bbr^d$ be an \emph{analytic nondegenerate map} defined on an interval $I_0\subq\R$. By definition, this means that the coordinate functions $f_1,\dots,f_d$ are analytic and together with the constant function $1$ are linearly independent over $\bbr$. The map $\vv f$ should be understood as the parameterisation of a curve $\cC$ in $\R^d$, namely $\cC=\vv f(I_0)$. In this case, the set $\bf^{-1}\left(\badw\right)$ precisely consists of the parameters $x\in I_0$ for which the corresponding point $\vv f(x)$ on the curve $\cC$ is badly approximable with respect to the weight $\bw$. For $d=2$ Badziahin and Velani \cite{BadziahinVelani2} proved that $\bf^{-1}\left(\badw\right)$ is
Cantor winning for every $\bw\in\calw_2$. This property was then improved to `winning' by An, Velani and the first named author \cite{ABV}. In fact, the `winning' property can be strengthened to `absolute winning' on applying \cite[Appendix B]{N13}, see also \cite[Remark 7]{ABV}. For higher dimensions, the first named author \cite{Beresnevich_BA}
proved that for every $\bw\in\calw_d$ the set $\bf^{-1}\left(\badw\right)$ is Cantor winning (see also \cite[Theorem B]{BadziahinHarrap}). This result was then improved by the third named author \cite{Yang} in the following manner.
By Definition~\ref{def:cantorWinning}, a Cantor winning set in $\bbr^d$ is $\alpha$-Cantor winning for
some $0\leq \alpha < d$. In \cite{Beresnevich_BA} the parameter $\alpha$ depends on $\bw$,
while in \cite{Yang} it was shown that $\bf^{-1}\left(\badw\right)$ is $\alpha$-Cantor winning
for some $0\leq \alpha < d$ that depends only on $d$. Eventually, the argument of \cite{BNY1}
strengthened the conclusions of \cite{Yang} to completely remove the dependence of $\alpha$ on $d$.
While it does not do it explicitly, it does so essentially by allowing the Cantor potential game to be played on
the support of any Ahlfors regular measure on $I_0$.
By \cite[Theorem 1.5]{BHNS} this implies that $\bf^{-1}\left(\badw\right)$ is absolute winning.
%\red{SHALL WE ADD A COMMENT ON 2 ACTIONS AND HOW THE SECOND ACTION IS DIFFERENT TO THAT IN THE CURVES PAPER?}

\subsection*{{Organisation of this paper}}
{In} Section~\ref{sec:schmidt} {we} recall the relevant
variants of Schmidt's game, definitions in fractal measure theory and establish the equivalence
between Theorem~\ref{thm:HAW} and Theorem~\ref{thm:intersection}. {In} Section~\ref{sec:daniCorrespondence}
{we} recall the Dani correspondence and the Kleinbock, Lindenstrauss, Weiss quantitative nondivergence.
The proof of Theorem~\ref{thm:intersection} is finally given in Section~\ref{sec:proof}.

\subsection*{Notation and conventions\label{conventions}.} Throughout this paper, we {will use} the following notation.
{Given a} metric space $(X,\dist)$, any $S\subq X$ {and} $r>0$, we denote the closed $r$ neighborhood of $S$
by
\[
B(S,r) := \left\{x\in X\sep \dist(x,S) \leq r\right\}.
\]
A closed ball $B(x_0,r):= \left\{x\in X\sep \dist(x,x_0) \leq r\right\}$
is defined by a fixed centre $x_0$ and radius $r>0$, although these in general are not uniquely 
determined by the ball as a set. In view of the latter, when referring to a ball $B$ we will mean the
 pair of its centre and radius and, with some abuse of notation, $B$ will also mean the corresponding set of 
 points when appearing in set theoretic expressions. The same will apply to the more general notion of $r$ neighborhood of a set $S$.
Finally, for any $\bx\in\bbr^d$, $r>0$ {and} $c>0$, {we let}
\[
cB\left(\bx,r\right) := B\left(\bx,cr\right).
\]

\section{Schmidt games and intersections with fractals}\label{sec:schmidt}

Schmidt's game is a quantitative version of the Banach-Mazur game played on a complete
metric space. Its corresponding winning sets are dense and often have large Hausdorff
dimension. Moreover, by definition, the collection of all $\alpha$-winning sets is stable under taking countable intersections, where $\alpha\in(0,1)$ is a certain parameter of Schmidt's games. Schmidt's winning sets are also stable under affine transformations, although the parameter $\alpha$ may change. Schmidt's game was introduced in \cite{schmidt1} and used to strengthen and simplify earlier results in Diophantine approximation. There are several modifications of Schmidt's game resulting in alternative notions of winning sets. These include the notions of absolute winning sets \cite{McMullen_absolute_winning}, HAW sets \cite{BFKRW} and Cantor winning sets \cite{BadziahinHarrap}. For a detailed survey of the various winning sets, their properties and the connections between them, see \cite{BHNS} and \cite{BFKRW}.
In particular, in \cite{BFKRW} it is proved that HAW sets are $\alpha$-winning for any $0<\alpha<\frac12$.

\begin{defn}[Hyperplane absolute winning game and sets, {\cite[\S2]{BFKRW}}]\label{def:hyperplaneWinning+}
The \emph{hyperplane absolute game} on $\bbr^d$  is played by two players, say Alice and Bob, who take turns making their moves. Bob starts by choosing a {\em parameter} $0<\beta<1/3$, which is fixed throughout the game, and a ball $B_0\subq \bbr^d$ of radius $r_0>0$. Subsequently for $n=0,1,2,\dots$, first, Alice chooses a neighborhood $A_{n+1}$ of any hyperplane in $\R^d$ of radius $\eps r_n$ for some $0<\eps\le \beta$; and second, Bob chooses a ball $B_{n+1} \subq B_n\setminus A_{n+1}$ of radius $r_{n+1}\ge\beta r_n$, where $r_n$ is the radius of $B_n$.
\par\noindent A set $S\subq \bbr^d$ is called \emph{hyperplane absolute winning \emph{(abbr. \emph{HAW})}} if Alice has a strategy which ensures that $S\cap \bigcap_{n\geq0}B_n\neq\varnothing$.
\end{defn}

For the purposes of this paper it will be convenient to use the following modified version of the hyperplane absolute game.

\begin{defn}\label{def:hyperplaneWinning}
The \emph{restricted hyperplane absolute game} on $\bbr^d$  is played by two players, say Alice and Bob, who take turns making their moves. Bob starts by choosing a {\em parameter} $0<\beta<1$, which is fixed throughout the game, and a ball $B_0\subq \bbr^d$ of radius $r_0>0$. Subsequently for $n=0,1,2,\dots$, first, Alice chooses a neighborhood $A_{n+1}$ of some hyperplane in $\R^d$ of radius $\beta r_n=\beta^{n+1}r_0$; and second, Bob chooses a ball $B_{n+1}\subq B_n\setminus A_{n+1}$ of radius $r_{n+1}=\beta r_n=\beta^{n+1}r_0$, where $r_n$ is the radius of $B_n$. If there is no such ball the game stops and Alice wins by default. Otherwise, the \emph{outcome} of the game is the unique point in $\bigcap_{n\geq0}B_n$.
\par\noindent A set $S\subq \bbr^d$ will be called \emph{restricted hyperplane absolute winning} if Alice has a strategy which
ensures that she either wins by default or the outcome lies in $S$.
\end{defn}

We note that there is no difference between HAW sets and restricted HAW sets and therefore throughout the rest of paper we will refer to the restricted hyperplane absolute game as the hyperplane absolute game, and to restricted hyperplane absolute winning sets as HAW sets. This fact was previously shown in \cite{FSU4} in relation to the absolute game. Formally, we have the following statement.

\begin{prop}\label{prop:modifiedHAW}
Let $S\subq\R^d$. Then $S$ is restricted hyperplane absolute winning if and only if $S$ is hyperplane absolute winning.
\end{prop}

The proof of this proposition is essentially the same as that of Proposition~4.5 in \cite{FSU4}. Indeed, the only change that is needed to obtain the proof of Proposition~\ref{prop:modifiedHAW} from the proof of Proposition~4.5 in \cite{FSU4} is to replace neighborhoods of balls, which are legal moves in the game played in \cite[Proposition~4.5 ]{FSU4} by neighborhoods of hyperplanes, which are Alice's legal moves in Definitions~\ref{def:hyperplaneWinning+} and \ref{def:hyperplaneWinning}. However, as Proposition~\ref{prop:modifiedHAW} forms a step in the argument towards our final goal, we give it a complete formal proof in the appendix at the end of this paper.

\medskip

In order to reduce Theorem~\ref{thm:HAW} to Theorem~\ref{thm:intersection}, let us recall the definitions of Ahlfors regular measures and absolute decaying measures, which can be found, for instance, in \cite{BFKRW}.

\begin{defn}\label{def:ahlfors}
Let $X$ be a metric space. Given $\alpha>0$, a Borel measure $\mu$ on $X$ is \emph{$\alpha$-Ahlfors
    regular} if there exist $A, \rho_0>0$ such that for every $\bx\in\supp\mu$
    \begin{equation}
    A^{-1}r^{\alpha} \leq \mu\left(B\left(\bx,r\right)\right) \leq Ar^{\alpha}\qquad\text{for all $0<r\leq \rho_0$}\,.\label{eq:AhlforsRegular}
    \end{equation}
    We say that $\mu$ is \emph{Ahlfors regular} if it is $\alpha$-Ahlfors regular for some $\alpha >0$.
\end{defn}

\begin{defn}\label{def:absolutelyDecaying}
    A Borel measure $\mu$ on $\bbr^{d}$ is called \emph{absolutely decaying}
    if there exist $D,\delta>0$ and $r_0>0$ such that for every $\bx\in \support\mu$, $0<r\leq r_0$,
    every hyperplane $H\subq\bbr^{d}$ and $r'>0$ we have that
    \begin{equation}\label{eq:absolutelyDecaying}
    \mu\left(B\left(H,r'\right)\cap B\left(\bx,r\right)\right) \leq D\left(\frac{r'}{r}\right)^{\delta}\mu\left(B\left(\bx,r\right)\right).
    \end{equation}
    \end{defn}

The following proposition allows us to reduce Theorem~\ref{thm:HAW} to Theorem~\ref{thm:intersection}. This proposition is already hinted in \cite[Remark 4.5]{BHNS}.

\begin{prop}\label{prop:known}
If $S\subq\bbr^d$ is HAW then $S\cap\support\mu\neq\varnothing$ for any Ahlfors regular absolutely decaying measure $\mu$ on $\bbr^d$.
Conversely, if $S$ is Borel and $S\cap\support\mu\neq\varnothing$ for any compactly supported Ahlfors regular absolutely decaying measure $\mu$ on $\bbr^d$, then $S\subq\bbr^d$ is HAW.
\end{prop}

Proposition~\ref{prop:known} has the following equivalent formulation, stated as Proposition~\ref{prop:equiv}, which does not use measures and is slightly easier to prove. First, recall the following definition appearing in \cite{BFKRW}.

\begin{defn}\label{def:hyperplane diffuse}
A nonempty closed subset $K\subq\bbr^{d}$ is called \emph{hyperplane diffuse}
if there exists $\beta>0$ and $r_0>0$ such that for every $\bx\in K$, $0<r\leq r_0$ and
every hyperplane $H\subq\bbr^{d}$ we have that
\begin{equation}\label{eq:hyperplane diffuse}
K\cap\left(B\left(\bx,r\right)\setminus B\left(H,\beta r\right)\right) \neq\varnothing\,.
\end{equation}
\end{defn}

\begin{prop}\label{prop:equiv}
If $S\subq\bbr^d$ is HAW then $S\cap K\neq\varnothing$ for any hyperplane diffuse set $K\subq\bbr^d$. Moreover, if $S$ is Borel then the converse also holds.
\end{prop}

Only the second parts of Propositions~\ref{prop:known} and \ref{prop:equiv} are new. Indeed, \cite[Proposition 5.5]{BFKRW} 
further proves a lower bound on the dimension of the intersection of HAW sets with the support of decaying measures, 
and an analogous statement about the intersection of HAW sets with hyperplane diffuse sets follows from \cite[Proposition 5.5]{BFKRW}. 
A direct proof of the first implication in Proposition~\ref{prop:equiv} is provided below for the reader's convenience.

The equivalence between Propositions~\ref{prop:known} and \ref{prop:equiv} follows from the fact that if $\mu$ is absolutely decaying then $\support\mu$ is hyperplane diffuse \cite[Proposition 5.1]{BFKRW}, and, on the other hand, if $K$ is hyperplane diffuse then there exists an Ahlfors regular absolutely decaying measure $\mu$ for which $\support\mu\subq K$. The latter can be shown on modifying the proof of \cite[Proposition~5.5]{BFKRW}, where an absolutely decaying measure $\mu$ is constructed. Formally, we have the following statement.

\begin{prop}\label{measure}
Let $K\subq\bbr^{d}$ be hyperplane diffuse. Then there exists a compactly supported absolutely decaying Ahlfors regular measure $\mu$ on $\bbr^d$ such that $\supp\mu\subq K$.
\end{prop}

To summarise above discussion, in order to fully justify our claim that Theorem~\ref{thm:HAW} follows from Theorem~\ref{thm:intersection}, it remains to give formal proofs to
Propositions~\ref{prop:equiv} and \ref{measure}. To begin with, we deal with the former, and start by stating two
auxiliary statements that will be used in the proof of Proposition~\ref{prop:equiv}.

\begin{lem}\label{lem:5.6}
For any $\beta>0$ there exists $0<\beta'<\beta$ and $N$ such that, for every ball $B=B(\bx,\rho)\subq\bbr^d$ there is a collection of at most $N$ hyperplanes $\calh_B$ such that for any hyperplane $H'$ there exists $H\in\calh_B$ for which
\[
B(\bx,\rho)\cap B\left(H',\beta'\rho\right)\subq B\left(H,\beta \rho\right).
\]
\end{lem}

\begin{proof}
The statement of this lemma is a specific case of Assumption C.6 in \cite{FSU4}, where $\beta'=\frac{\beta}{3}$. In the case of hyperplanes (Lemma~\ref{lem:5.6}) it is verified as part (2) of Observation C.7. in \cite{FSU4}.
\end{proof}

The following is a slightly simplified version of Lemma 4.3 in \cite{BFKRW}.

\begin{lem}\label{lem:BFKRW4.3}
Let $K\subq\bbr^d$ be hyperplane diffuse. Then there exist %sufficiently small constants
$0<\beta_0<\frac{1}{3}$ and $r_0>0$ such that for any $0<r\le r_0$, any $\bx\in K$ and any hyperplane $H$ there exists $\bx'\in K$ such that
\begin{equation}\label{eq:BFKRW4.3}
B\left(\bx',\beta_0 r\right)\subq B(\bx,r)\setminus B\left(H,\beta_0 r\right).
\end{equation}
\end{lem}

\begin{proof}[Proof of Proposition~\ref{prop:equiv}]
Assume that $S$ is hyperplane absolute winning and $K$ is hyperplane diffuse. Let $\beta_0$ and $r_0$ be the same as in Lemma~\ref{lem:BFKRW4.3} and 
suppose that Alice and Bob play the restricted hyperplane absolute game according to Definition~\ref{def:hyperplaneWinning}. Suppose 
that on the first move Bob chooses $\beta=\beta_0$ and a ball $B_0$ of radius $r_0$ centred in $K$. These are valid assumptions 
since $\beta_0\in(0,1)$, $K$ is non-empty and $B_0$ can be arbitrary. Let $n\ge0$ and suppose that $B_0,\dots,B_n$ are the balls 
arising from the restricted hyperplane absolute game that Alice and Bob play with Alice using her winning strategy. Suppose that 
all these balls are centred in $K$. Note that, by the choice of $B_0$, this is true for $n=0$. Let $A_{n+1}$ be the neighborhood of any 
hyperplane in $\R^d$ of radius $\beta^{n+1}r_0$ that Alice chooses according to Definition~\ref{def:hyperplaneWinning}. Then, by Lemma~\ref{lem:BFKRW4.3},
Bob can choose a ball $B_{n+1}$ of radius $\beta^{n+1}r_0$ which is contained in $B_n\setminus A_{n+1}$ and centered in $K$. Indeed, $B_{n+1}$ can be defined to be $B\left(\bx',\beta_0 r\right)$ arising from Lemma~\ref{lem:BFKRW4.3} with $B(\bx,r)=B_n$ and $B\left(H,\beta_0 r\right)=A_{n+1}$. Thus, for any $n$ Bob has a legal move, which means that Alice cannot win by default, and the sequence $\left(B_n\right)_{n\ge0}$ can be made infinite. Furthermore, as we have shown above Bob can play so that the centres of $B_n$ are all in $K$ and so the unique point in $\bigcap_{n\geq0}B_n$ lies in $K$.
At the same time Alice can play using her winning strategy so that the unique point in $\bigcap_{n\geq0}B_n$ also lies in $S$. Therefore, $S\cap K\neq\varnothing$ and this completes the proof of the first part of Proposition~\ref{prop:equiv}.

For the converse, assume that $S$ is a Borel set which is not HAW. Then, by Borel determinacy theorem for the absolute game appearing in \cite[Theorem 1.6]{FLS}, Bob has a winning strategy, which will be fixed for the rest of the proof.
Let $\beta$ and $B_0$ be chosen on the first move of Bob according to his winning strategy. Define $\calb_0=\{B_0\}$ and continue by induction to construct collections of closed balls $\calb_n$ as follows. Given $\calb_n$, for every $B\in \calb_n$ let $\calh_B$ be the collection of hyperplanes arising from Lemma~\ref{lem:5.6}. Define
%$\calb_{n+1}=\bigcup_{B\in\calb_n}\calb_{n+1}\left(B\right)$, where
$\calb_{n+1}(B)$ to be the collection of all of Bob's responses according to the winning strategy while considering {the} hyperplanes in $\calh_B$ as possible moves of Alice. {Note that $\calb_{n+1}(B)$ is always nonempty.} Define
\begin{equation}\label{K}
K=\bigcap_{n\geq0}\bigcup_{B\in\calb_n}B\,.
\end{equation}
%{By the compactness of $\{1,\ldots,N\}^\bbn$,} where $N$ is as in Lemma~\ref{lem:5.6},
{By Lemma~\ref{lem:5.6},} every $\bx\in K$ is an outcome of the hyperplane absolute game played according to Bob's winning strategy. {Therefore,} {$\bx\not\in S$}. {Since $\bx$ is an arbitrary point of $K$, we have} that $K\cap S=\varnothing$.

It is left to verify that $K$ is hyperplane diffuse. Indeed, we will show that it is $\frac{\beta'\beta}{2}$ hyperplane diffuse for $\beta'$ as in Lemma~\ref{lem:5.6}. Assume $\bx\in K$, $0<r\leq r_0$ and $H'\subq\bbr^{d}$ is a hyperplane, where $r_0$ is the radius of $B_0$. Let $n$ be the unique {positive} integer such that
\begin{equation}\label{eq:scale}
2\beta^{n}r_0\leq r <2\beta^{n-1}r_0\,,
\end{equation}
which clearly exists since $0<\beta<1$.
Since $\bx\in K$, by \eqref{K}, there exists a ball $B=B\left(\bx_0,\beta^{n}r_0\right)\in\calb_n$ such that $\bx\in B$.
The left hand side of \eqref{eq:scale} implies that $B\subq B(\bx,r)$. By Lemma~\ref{lem:5.6} applied with $\rho=\beta^n r_0$, there exists $H\in \calh_B$ such that
\[
B\cap B\left(H',\beta'\beta^{n}r_0\right)\subq B\left(H,\beta^{n+1} r_0\right).
\]
The right hand side of \eqref{eq:scale} implies that $\frac{\beta\beta'}{2}r<\beta'\beta^{n}r_0$ and hence
\[
B\cap B\left(H',\frac{\beta\beta'}{2}r\right)\subq B\left(H,\beta^{n+1} r_0\right).
\]
By the definition of $\calb_{n+1}(B)$, there exists a ball $B'\in\calb_{n+1}(B)$ such that $B'\cap B\left(H,\beta^{n+1} r_0\right)=\varnothing$. Since the collections $\calb_{n+1}(B)$ are always nonempty, by \eqref{K}, we have that $K\cap B'\neq\varnothing$.
Since $\varnothing \neq K\cap B'\subq K\cap B\subq K \cap B(\bx_0,r)$, we have $K \cap B(\bx_0,r)\not\subq B\left(H',\frac{\beta\beta'}{2}r\right)$.  Hence, $K \cap \left(B(\bx_0,r)\setminus B\left(H',\frac{\beta\beta'}{2}r\right)\right)\neq\varnothing$. This verifies Definition~\ref{def:hyperplane diffuse} for the set $K$ and thus completes the proof.
\end{proof}

\subsection{Proof of Proposition~\ref{measure}}

The proof of Proposition~\ref{measure} relies on a standard construction of Ahlfors regular measures in $\bbr^d$ via decreasing collections of disjoint balls. For this construction we follow \cite[Section 7.2]{KleinbockWeiss1}. Assume that $0<\beta<1$, $r_0 > 0$, and that $N>1$ is some fixed integer. Assume that $B_0$ is a closed ball {in $\R^d$} and that $\left(\calb_n\right)_{n\geq0}$ is a sequence of collections of closed balls such that $\calb_0 = \left\{B_0\right\}$, any $B\in\calb_n$ is a ball of radius $\beta^nr_0$, and
\begin{equation}\label{vb999}
{\calb_{n+1}=\bigcup_{B\in\calb_n}\calb_{n+1}\left(B\right),}
\end{equation}
where for every $B\in\calb_n$ the collection
\[
\calb_{n+1}\left(B\right) = \left\{B'\in\calb_{n+1}\sep B'\subq B\right\}
\]
contains exactly $N$ disjoint balls for any $n\geq0$. Define
\begin{equation}\label{KK}
K = \bigcap_{n\geq0}\bigcup_{B\in\calb_n}B\,.
\end{equation}
Also define the sequence of probability measures
\[
\mu_n = \frac{1}{\#\calb_n}\sum_{B\in\calb_n}\lambda|_{B}\qquad(n\geq0)
\]
where $\lambda$ is the Lebesgue measure on $\bbr^d$ and $\lambda|_B$ is the normalised restriction of $\lambda$ to $B$ which is defined by the formula $\lambda|_B(A) = \lambda(A\cap B)/\lambda(B)$ for any Lebesgue measurable set $A$.
{By \eqref{vb999}, we have that $\supp\mu_n\subp\supp\mu_{n+1}$ for any $n\ge0$.} Let $\mu$ be the weak limit of $\mu_n$ and let
\begin{equation}\label{eq:exponent}
\alpha = - \frac{\log N}{\log \beta}\,.
\end{equation}
{Note that for every $n\ge0$ and every $B\in\calb_n$ we have that
$$
\mu(B)=\mu_n(B)=\frac{1}{\#\calb_n}=N^{-n}\,.
$$}

\begin{prop}\label{prop:treeLike}
Let $K\subq\bbr^d$, $\mu$ and $\alpha$ be defined as above. Then $\support \mu = K$ and $\mu$ is $\alpha$-Ahlfors regular.
\end{prop}

Proposition~\ref{prop:treeLike} is proved in \cite[Proposition 7.1]{KleinbockWeiss1} for $\bbr^d$ with the supremum norm. For completeness we repeat their proof with the Euclidean norm.

\begin{proof}
Assume $\bx\in K$ and $0 < r \leq 2r_0$. Let $n$ be the unique integer for which
\begin{equation}\label{eq:KleinbockWeiss1}
2\beta^{n+1}r_0 < r \leq 2\beta^nr_0\,.
\end{equation}
Since $\bx\in K$, {by \eqref{KK} and the disjointness of the balls in $\calb_{n+1}$,} there exists a unique ball $B\in\calb_{n+1}$ such that $\bx\in B$. The left hand side of \eqref{eq:KleinbockWeiss1} implies that $B\subq B(\bx,r)$. So, by \eqref{eq:exponent} and the right hand side of \eqref{eq:KleinbockWeiss1} this implies that
\[
\mu(B(\bx,r))\geq \mu(B) = \frac{1}{N^{n+1}} = \beta^{\alpha(n+1)} \geq \left(\frac{\beta}{2r_0}\right)^\alpha r^\alpha\,.
\]
On the other hand, by the right hand side of \eqref{eq:KleinbockWeiss1}, there exists a constant $M\geq1$ depending only on $\beta$ and $d$ such that
\[
\#\left\{B\in\calb_n\sep B\cap B(\bx,r)\neq\varnothing\right\} \leq M\,.
\]
Therefore, by \eqref{eq:exponent} and the left hand side of \eqref{eq:KleinbockWeiss1} this implies that
\[
\mu(B(\bx,r))\leq \frac{M}{N^n} = M\beta^{\alpha n} < M\left(\frac{1}{2r_0\beta}\right)^\alpha r^\alpha\,.
\]
So \eqref{eq:AhlforsRegular} is verified with $A=\max\left\{\left(\frac{2r_0}{\beta}\right)^\alpha,M\left(\frac{1}{2r_0\beta}\right)^\alpha\right\}$.

\end{proof}

The proof of Proposition~\ref{measure} is based on the construction described above, for a particular choice of balls which stay far from appropriate neighborhoods of hyperplanes in each level. The argument used for the proof of \cite[Proposition 5.5]{BFKRW} provides such a choice. It is based on the following lemma.

\begin{defn}\label{def:generalPosition}
Say that $d$ points in $\bbr^d$ are in \emph{general position} if they lie on a unique hyperplane. If $\bx_1,\ldots,\bx_d\in\bbr^d$ are in general position denote this hyperplane by $H\left(\bx_1,\ldots,\bx_d\right)$.
\end{defn}

\begin{lem}[{See {\cite[Lemma 5.6]{BFKRW}}}]\label{lem:lemma5.6}
Given $\beta_0 > 0$, there exists a positive parameter $\beta'\leq\beta_0$ such that for every $\bx\in\bbr^d$, $\rho>0$, and $\bx_1,\ldots,\bx_d\in B(\bx,\rho)$ in general position such that the balls $B\left(\bx_i,\beta_0\rho\right)$ are contained in $B(\bx,\rho)$ for every $1\leq i \leq d$ and are pairwise disjoint, if a hyperplane $H$ intersects $B\left(\bx_i,\beta'\rho\right)$ for every $1\leq i\leq d$ then
\[
B(\bx,\rho)\cap B\left(H,\beta'\rho\right)\subq B\left(H\left(\bx_1,\ldots,\bx_d\right),\beta_0\rho\right).
\]
\end{lem}

Lemma~\ref{lem:lemma5.6} is stated in \cite{BFKRW} with the general position assumption implicit. We repeat the proof that appears in \cite{BFKRW} for completeness.

\begin{proof}
Without loss of generality assume that $\bx=0$ and $\rho=1$. By contradiction, assume that for every integer $k\geq1$ there are $\bx_{1,k},\ldots,\bx_{d,k}{\in B(0,1)}$ in general position and a hyperplane $H_k$ that intersects $B\left(\bx_{i,k},\frac{1}{k}\right)$ for each $1\leq i\leq d$ but
\begin{equation}\label{eq:contradictionHyperplane}
B(0,1)\cap B\left(H_k,\frac{1}{k}\right) \not\subq B\left(H(\bx_{1,k},\ldots,\bx_{d,k}),\beta_0\right).
\end{equation}
By the compactness of $B(0,1)$ there are subsequences $\left(\bx_{1,k_j},\ldots,\bx_{d,k_j}\right)$ and $H_{k_j}$ that converge, say to $\left(\bx_1,\ldots,\bx_d\right)$ and $H$ respectively. Then necessarily $\bx_1,\ldots,\bx_d\in H$ and, therefore, any $j$ large enough satisfies
\begin{align*}
B(0,1)\cap B\left(H_{k_j},\frac{\beta_0}{3}\right) & \subq B(0,1)\cap B\left(H,\frac{2\beta_0}{3}\right)\\
 & \subq B(0,1)\cap B\left(H\left(\bx_{1,k_j},\ldots,\bx_{d,k_j}\right),\beta_0\right).
\end{align*}
%Applying \eqref{eq:firstHyperplane} with $\eps=\frac{\beta_0}{4}$ and then \eqref{eq:secondHyperplane} with $\eps=\frac{\beta_0}{2}$, and
Choosing $j$ large enough so that $\frac{1}{k_j}\leq \frac{\beta_0}{3}$ we obtain a contradiction to \eqref{eq:contradictionHyperplane}.
\end{proof}

\begin{proof}[Proof of Proposition~\ref{measure}]
We follow the proof of \cite[Proposition 5.5]{BFKRW}. Assume $K$ is hyperplane {diffuse}. The goal is to construct an Ahlfors regular absolutely decaying measure supported on a subset of $K$. Let $\beta_0$ and $r_0$ be as in Lemma~\ref{lem:BFKRW4.3}. Let $\beta'$ be as in Lemma~\ref{lem:lemma5.6}, and let
\begin{equation}\label{eq:beta'}
\beta = \frac{\beta'}{2}\,.
\end{equation}
Let $\bx_0\in K$ be any point, and set $B_0=B(\bx_0,r_0)$ and $\calb_0=\left\{B_0\right\}$. Recursively construct the collections $\calb_{n+1}(B)$ for every integer $n\geq0$ and every $B\in\calb_n$ as follows. Construct by recursion a collection of $d+1$ points in $K\cap B$. Assume $\bx_1,\ldots,\bx_i\in K\cap B$ are already defined for some $0\leq i \leq d$, and let $H$ be any hyperplane that passes through $\bx_1,\ldots,\bx_i$. By \eqref{eq:BFKRW4.3} there exists a point $\bx_{i+1}$ such that
\begin{equation}\label{eq:measureInduction}
B\left(\bx_{i+1},\beta_0\beta^{n} r_0\right) \subq B \setminus B\left(H,\beta_0\beta^{n} r_0\right).
\end{equation}
Define $\calb_{n+1}(B)=\left\{B\left(\bx_1,\beta^{n+1}r_0\right),\ldots,B\left(\bx_{d+1},\beta^{n+1}r_0\right)\right\}$. Since $\beta<\beta_0$ this is a collection of $d+1$ disjoint balls contained in $B$. Let $\mu$ be as defined in the beginning of this section. Then $\support\mu\subq K$ since for every $n\geq0$ every $B\in\calb_n$ is a ball centered in $K$. Proposition~\ref{prop:treeLike} guarantees that $\mu$ is Ahlfors regular. It is left to verify that $\mu$ is absolutely decaying.

Assume $r\leq r_0$, $\bx\in\support\mu$ and $r' > 0$, and let $H$ be any hyperplane. Let $n\geq0$ be the unique integer satisfying
\begin{equation}
2\beta^{n+1} r_0 	 \leq r  	< 2\beta^{n} r_0\,. \label{eq:measureConstructionBall}
\end{equation}

Since $\bx\in\support\mu$ there are balls $B\subq B'$ with $B\in\calb_{n+1}$ and $B'\in\calb_n$ such that $\bx\in B$. The left hand side of \eqref{eq:measureConstructionBall} implies $B\subq B(\bx,r)$. On the other hand, equation \eqref{eq:measureInduction} implies that
\[
\dist\left(B',B''\right) \geq 2\left(\beta_0 - \beta\right)\beta^{n-1}r_0
\]
for any $B'\neq B''\in\calb_n$, therefore, since $\beta=\frac{\beta'}{2}\leq\frac{\beta_0}{2}$, the right hand side of \eqref{eq:measureConstructionBall} implies that $B(\bx,r)\cap B''=\varnothing$ for any $B'\neq B''\in\calb_n$. So, $B(\bx,r)\cap\support\mu\subq B'$.

It is enough to verify \eqref{eq:absolutelyDecaying} for every $r'$ small enough. Assume that $r' < \frac{1}{2}\beta r$ and let $m\geq1$ be the unique integer satisfying
\begin{equation}
\frac{1}{2}\beta^{m+1}r 	 \leq r' 	<  \frac{1}{2}\beta^{m}r\,. \label{eq:measureConstructionHyperplane}
\end{equation}
The right hand side of both \eqref{eq:measureConstructionBall} and \eqref{eq:measureConstructionHyperplane} imply that $r' < \beta^{m + n}r_0$. Therefore, for every $1\leq k\leq m$ and every $B''\in\calb_{n+k-1}$, the hyperplane neighborhood $B(H,r')$ intersects at most $d$ balls from $\calb_{n+k}(B'')$. Indeed, recall that
\[
\calb_{n+k}(B'')=\left\{B\left(\bx_1,\beta^{n+k}r_0\right),\ldots,B\left(\bx_{d+1},\beta^{n+k}r_0\right)\right\}.
\]
If $B(H,r')\cap B\left(\bx_i,\beta^{n+k}r_0\right)\neq\varnothing$ for every $1\leq i\leq d+1$ then \eqref{eq:beta'} implies that
\[%\begin{equation}\label{eq:hidden}
H\cap B\left(\bx_i,\beta'\beta^{n+k-1}\right)\neq\varnothing
\]%\end{equation}
for every $1\leq i\leq d+1$. By construction, the points $\bx_1,\ldots,\bx_{d}$ are in general position, so Lemma~\ref{lem:lemma5.6} gives
\[%\begin{equation}\label{eq:hidden2}
B''\cap B\left(H,\beta'\beta^{n+k-1}r_0\right) \subq B\left(H(\bx_1,\ldots,\bx_d),\beta_0\beta^{n+k-1}r_0\right).
\]%\end{equation}
In particular,
$\bx_{d+1}\in B\left(H(\bx_1,\ldots,\bx_d),\beta_0\beta^{n+k-1}r_0\right)$, which contradicts \eqref{eq:measureInduction}. The upshot is that $B(H,r')$ intersects at most $d^m$ balls in $\calb_{n+m}$, therefore,
\begin{align*}
\mu\left(B(\bx,r)\cap B\left(H,r'\right)\right) & \leq \left(\frac{d}{d+1}\right)^m\mu\left(B'\right) \\
& = \left(\frac{d}{d+1}\right)^m(d+1)\mu\left(B\right)  \leq \left(\frac{d}{d+1}\right)^m(d+1)\mu\left(B(\bx,r)\right).
\end{align*}
By the left hand side of \eqref{eq:measureConstructionHyperplane}, this verifies \eqref{eq:absolutelyDecaying} with $\delta = \frac{\log\frac{d}{d+1}}{\log \beta}$ and $D=\left(\frac{2}{\beta}\right)^\delta(d+1)$.

\end{proof}

\subsection{Cantor potential game}

In order to prove Theorem~\ref{thm:intersection}, we will use the Cantor potential game
introduced in \cite{BHNS}. The game and its corresponding winning sets are defined
as follows.

\begin{defn}\label{def:cantorWinning}
Let $X$ be a complete metric space and $\alpha\geq0$. The \emph{$\alpha$-Cantor potential game}
is played by two players, say Alice and Bob, who take turns making their moves. Bob starts by choosing a parameter $0<\beta<1$, which is fixed throughout the game, and a ball $B_0\subq \bbr^d$ of radius $r_0>0$. Subsequently for $n=0,1,2,\dots$, first, Alice chooses collections  $\mathcal{A}_{n+1,i}$
of at most $\beta^{-\alpha(i+1)}$ balls of radius $\beta^{n+1+i}r_0$ for every $i\geq0$.
Then, Bob chooses a ball $B_{n+1}$ of radius $\beta^{n+1}r_0$ which is contained in $B_n$ and
disjoint from $\bigcup_{0\leq \ell \leq n}\bigcup_{A\in\mathcal{A}_{n+1-\ell,\ell}}A$. If there is no
such ball the game stops and Alice wins by default. Otherwise, the \emph{outcome} of the game is the
unique point in $\bigcap_{n\geq0}B_n$.

A set $S\subq X$ is called \emph{$\alpha$-Cantor winning} if Alice has a strategy which ensures that she
either wins by default or the outcome lies in $S$. If $X$ is the support of an $\alpha$-Ahlfors regular measure then $S\subq X$ is called \emph{Cantor winning} if it is $\alpha'$-Cantor winning for some $0\leq \alpha'<\alpha$.
\end{defn}

It is proved in \cite{BHNS} that this definition of $\alpha$-Cantor winning sets agrees with the
original definition given in \cite{BadziahinHarrap}. Here the convention regarding $\alpha$ is
opposite to the one used in \cite{BadziahinHarrap}. For example, in our convention $0$-Cantor winning sets are absolute winning, see \cite{BHNS} or \cite{BFKRW} for the definition of absolute winning. This convention allows the definition of the
$\alpha$-Cantor potential game to be independent of the space $X$. This comes at the price that some properties of
Cantor winning subsets do depend on $X$. We will use the following fact about Cantor winning sets.

\begin{thm}[See~{\cite[Theorems 3.4, 4.1]{BHNS}}]\label{thm:cantor-winning-nonempty}
Let $X$ be the support of an $\alpha$-Ahlfors regular measure and let $S \subq X$ be Cantor winning. Then $S\neq \varnothing$.
\end{thm}

We finish this section by stating an auxiliary lemma about \emph{efficient covers} for Ahlfors regular measures, which will be used in Section~\ref{sec:proof}.

\begin{lem}\label{lem:covering}
Let $\mu$ be an Ahlfors regular measure on {a metric space} $X$, let $A,\alpha,r_0$ be
as in Definition~\ref{def:ahlfors} and let $S\subq X$ be any subset.  {Suppose that $0<r\leq r_0$ and $\mu\left(B\left(S,r\right)\right)<\infty$. Then} there exists a cover of $S\cap\support\mu$
with balls of radius $3r$ of cardinality at most
\begin{equation}\label{eq:covering}
%\frac{A3^\alpha\mu\left(B\left(S,3r\right)\right)}{r^\alpha}.
\frac{A\mu\left(B\left(S,r\right)\right)}{r^\alpha}\,.
\end{equation}
\end{lem}

\begin{proof}
Assume $S\subq X$ and $r>0$. Without loss of generality assume that $S\neq\varnothing$. Note that $B(S,r)$ is a Borel set. Choose any collection of points $\calu\subq S\cap \support \mu$ such that
$\left\{B\left(x,r\right)\sep x\in\calu\right\}$ is a collection of pairwise disjoint balls.
By the pairwise disjointness and \eqref{eq:AhlforsRegular},
\begin{equation}\label{vb123}
\#\calu \times \frac{r^\alpha}{A} \leq \sum_{x\in\calu}\mu\left(B\left(x,r\right)\right) = \mu\left(\bigcup_{x\in\calu}B\left(x,r\right)\right) \leq \mu\left(B\left(S,r\right)\right).
\end{equation}
Since $\mu\left(B\left(S,r\right)\right)<\infty$, any such collection $\calu$ is finite. Therefore, of all the collections $\calu$ as above there exists one, say $\calu_{\max}$, with the maximal number of elements. Then, by its maximality, for any $x'\in S\cap\supp\mu$ there exists $x\in \calu_{\max}$ such that $\dist(x,x')\le 2r$. Then, $x'\in B(x,3r)$ and we conclude that $\left\{B\left(x,3r\right)\sep x\in\calu_{\max}\right\}$ is a cover of $S\cap\supp\mu$. Finally, \eqref{vb123} implies \eqref{eq:covering} and the proof is complete.
\end{proof}

\section{Homogeneous dynamics and quantitative nondivergence}\label{sec:daniCorrespondence}

The connection between Diophantine approximation and homogeneous dynamics is well known as the Dani correspondence.
In this context there is a beautiful relation between bounded orbits and badly approximable vectors. Throughout,
$\diag\left(b_1, \dots, b_d\right)$ denotes the $d\times d$ diagonal matrix with diagonal entries $b_1, \dots, b_d$.
\par Let $G := \sldr$ and $\Gamma := \sldz$. The homogeneous space $X_{d+1}:=G/\Gamma$ can be identified with
the moduli space of unimodular lattices in $\R^{d+1}$ via the following map:
\[ g\Gamma \in X_{d+1} \mapsto g\Z^{d+1}\,.\]
Given $\bw=(w_1,\dots,w_d)\in\calw_d$ and $b>1$, for any $n\in\bbz$ we let
\begin{equation}\label{eq:an}
a_{n}:=\left(\begin{array}{cccc}
b^{n}\\
& b^{-w_1n} \\
& & \ddots\\
&  & & b^{-w_dn}
\end{array}\right) \in G\,.
\end{equation}
Further, for any $\bx=(x_1,\dots,x_d)\in\bbr^d$ let
\begin{equation}\label{eq:ux}
u_{\bx}:=\left(\begin{array}{cccc}
1 &  x_1 & \cdots & x_d\\
& 1 & &  \\
& &  \ddots  & \\
&  &  & 1
\end{array}\right) \in G\,.
\end{equation}
For $\eps>0$ define the set
\[
K_{\eps}:=\left\{ \Lambda \in X_{d+1}\sep\left\|\bv\right\|\geq\eps\text{ for any }\bv\in\Lambda\setminus\left\{ 0\right\} \right\},
\]
where $\left\|\bv\right\|$ is the Euclidean norm of $\bv$.
Then, as is well known for any $\bx\in\bbr^d$ we have that
\begin{equation}\label{eq:daniCorrespondence}
\bx\in\bad\left(\bw\right)\;\iff\;\exists~\eps>0\text{ such that }a_{n}u_{\bx}\bbz^{d+1}\in K_{\eps}\text{ for every }n\in\bbn\,.
\end{equation}
See \cite[Appendix]{BPV} and \cite[Appendix A]{Beresnevich_BA} for detailed explanation of this equivalence.

Recall that by Mahler's criterion, the complements of the sets $K_\eps$ give a basis for the topology at $\infty$ in $X_{d+1}$,
 so \eqref{eq:daniCorrespondence} may be rephrased as $\bx\in\badw$
 if and only if $\left\{a_{n}u_{\bx}\bbz^{d+1}\sep n\in\bbn\right\}$ is bounded in $X_{d+1}$.

It is straightforward to verify that for every $\bx'\in\bbr^d$ we have that
\begin{equation}\label{vb101}
a_n u_{\bx'} a_n^{-1} = u_{\diag\left(b^{(1+w_1)n},\ldots,b^{(1+w_d)n}\right)\bx'}\,.
\end{equation}
Note that if $\bx=\bx_0+\bx'$ then $u_{\bx}=u_{\bx'}u_{\bx_0}$ and therefore
\[
a_{n}u_{\bx}\bbz^{d+1} = a_n u_{\bx'} a_n^{-1}a_n u_{\bx_0}\bbz^{d+1} ~\stackrel{\eqref{vb101}}{=}~ u_{\diag\left(b^{(1+w_1)n},\ldots,b^{(1+w_d)n}\right)\bx'} a_n u_{\bx_0}\bbz^{d+1}\,.
\]
Thus, on letting $\Lambda = a_n u_{\bx_0}\bbz^{d+1}$ and $\by=\diag\left(b^{(1+w_1)n},\ldots,b^{(1+w_d)n}\right)\bx'$ we see that the set of parameters $\by\in\bbr^{d}$ for which $u_{\by}\Lambda\in K_\eps$ plays a role in the study of bounded orbits of $u_{\bx}\bbz^{d+1}$ under the actions by $a_n$. The Dani-Kleinbock-Margulis quantitative nondivergence estimate (see \cite{Dani5, KleinbockMargulis2}) gives a sharp and uniform upper bound on the Lebesgue measure of the set of $\by$ for which $u_\by\Lambda\not\in K_\eps$ under some conditions on the lattice $\Lambda$. Later this was generalised to ``friendly'' measures by Kleinbock, Lindenstrauss and Weiss \cite{KLW}. Within this paper we will use the following direct consequence of Theorem~5.11 in \cite{BNY1}, which in turn is a consequence of the results of \cite{KLW}.

\begin{thm}\label{thm:KLW}
Assume $\mu$ is an Ahlfors regular absolutely decaying measure on $\bbr^{d}$.
Then for any $\bz\in \supp\mu$ there exists an open ball $B(\bz)$ centred at $\bz$ and constants $C,\gamma>0$ such that for any ball $B\subq B(\bz)$ centred in $\supp\mu$, any diagonal matrix $g\in G$ and any $0<\rho\le1$ at least one of the following two conclusions holds\/{\rm:}\\[-2ex]
\begin{enumerate}
\item[{\rm(i)}] for all $\varepsilon>0$
\begin{equation}\label{eq:KLW}
\mu\left(\left\{\bx\in B:gu_{\bx}\bbz^{d+1}\notin K_{\eps}\right\}\right)\le C\left(\frac{\varepsilon}{\rho}\right)^{\gamma}\mu(B)\,;
\end{equation}
\item[{\rm(ii)}] there exists $\vv0\neq\vv v=\vv v_1\wedge\cdots\wedge\vv v_j\in\bigwedge^j\left(\Z^{d+1}\right)$ with $1 \le j \le d$ such that
$$
\sup\limits_{x\in B}\|gu_{\bx}\vv v\|< \rho\,.
$$
\end{enumerate}
\end{thm}

In order to use Theorem~\ref{thm:KLW} some notation related to the action of $G$ on the exterior algebra of $\bbr^{d+1}$ is set up in the rest of this section.

Let $\be_{+}:=(1,0,\ldots,0)$ and $\be_i := (0,\ldots,1,\ldots,0)$, where for every $1\leq i\leq d$ the $i+1$st coordinate is one and the rest are zero, be the standard basis of $\bbr^{d+1}$. For any $I\subq\{+,1\ldots,d\}$
let $\be_I=\bigwedge_{i\in I}\be_i$ be the wedge product of basis elements with indices in $I$.
For any $1\leq j\leq d$ the collection $\left\{\be_I\sep\#I=j\right\}$ is a basis
of $\bigwedge^j\left(\bbr^{d+1}\right)$. Define an inner product on $\bigwedge^j\left(\bbr^{d+1}\right)$
by setting $\left\langle\be_I,\be_J\right\rangle=\delta_{I,J}$ (where $\delta_{I,J}:=1$ if $I= J$ and $\delta_{I,J}:=0$ otherwise) and extending linearly. Let $\|\cdot\|$ be the Euclidean norm which is derived from this inner product. Note that this notation is consistent with that of Theorem~\ref{thm:KLW}.

For every $1\leq j\leq d$ define the subspaces
$$
V_{+} := \linearspan_\bbr\{\be_{I}\sep +\in I\}\qquad\text{and}\qquad V_{-} := \linearspan_\bbr\{\be_{I}\sep I\subq\{1,\ldots,d\}\}\,.
$$
Each vector $\bv\in\bigwedge^j\left(\bbr^{d+1}\right)$ decomposes uniquely into $\bv=\bv_{+}+\bv_{-}$ with $\bv_{+}\in V_{+}$ and $\bv_{-}\in V_{-}$.

Let $G$ act on $\bigwedge^j\left(\bbr^{d+1}\right)$ by linear transformations defined on wedge products as follows: for any $g\in G$ and $\bv=\bv_1\wedge\ldots\wedge\bv_j\in\bigwedge^j\left(\bbr^{d+1}\right)$ we define
\begin{equation}\label{g_action}
g\bv = g\bv_1\wedge\ldots\wedge g\bv_j\,.
\end{equation}

\begin{prop}\label{prop:exterior}
Assume $\bx\in\bbr^d$, $h\in\Z$, $h\ge0$ and $\bv=\bv_1\wedge\ldots\wedge\bv_j\in\bigwedge^j\left(\bbr^{d+1}\right)$, $1\le j\le d$. Then, assuming that $\bv_i=v_{i,+}\be_++v_{i,1}\be_1+\dots+v_{i,d}\be_d$, we have that
\begin{align}
u_{\bx}\bv=\bv + \be_{+}\wedge\left(\sum_{i=1}^j (-1)^{i+1} \left(v_{i,1}x_1+\ldots+v_{i,d}x_d\right)\bigwedge_{i'\neq i}\bv_{i'}\right); \label{eq:uaction} \\
\|a_{-h}\bv_{+}\| \leq b^{-w_dh}\|\bv_{+}\|\quad\text{and }\quad \|a_{-h}\bv_{-}\| \leq b^{h}\|\bv_{-}\|\,, \label{eq:aaction}
\end{align}
where $w_d$ is assumed to be the smallest weight.
\end{prop}

\begin{proof}
Both \eqref{eq:uaction} and \eqref{eq:aaction} are elementary to prove. Indeed, \eqref{eq:uaction} is an immediate consequence of definition \eqref{g_action} and the easily verified equation $u_{\bx}\bv_i=\bv_i+\left(v_{i,1}x_1+\ldots+v_{i,d}x_d\right)\be_+$ together with the alternating property of the wedge product and, in particular, the fact that $\be_+\wedge\be_+=\vv0$. In turn, since the standard basis $\be_I$ of $\bigwedge^j\left(\R^{d+1}\right)$, where $I\subq\{+,1,\dots,d\}$ and $\#I=j$, is orthonormal and each of $\be_I$ is an eigenvector of $a_{-h}$, it suffices to verify \eqref{eq:aaction} for the basis vectors $\be_I$. The latter is a trivial job done by inspecting \eqref{eq:aaction}. We leave further computational details, which are straightforward, to the reader.
\end{proof}

When applying Theorem~\ref{thm:KLW} in Section~\ref{sec:proof} we will use the following simple bound.

\begin{lem}\label{thm:KLW3.3}
For every ball $B\subq\bbr^d$, every diagonal matrix $g=\diag \left(b_+,b_1,\ldots,b_d\right)$ such that $b_+\geq1$ and $0 < b_1,\ldots b_d \leq 1$ such that $b_+b_1\cdots b_d=1$, and every $\vv v=\vv v_1\wedge\cdots\wedge\vv v_j\in\bigwedge^j\left(\Z^{d+1}\right)$ with $1 \le j \le d+1$ such that $\bv\neq\vv0$ we have that
\begin{equation}\label{eq:KLW3.3}
\sup\limits_{x\in B}\|gu_{\bx}\vv v\|\ge \min\left\{1,r_B\right\},
\end{equation}
where $r_B$ is the Euclidean radius of $B$.
\end{lem}

\begin{proof}
The case of $j=d+1$ is trivial since in this case we have that $\|gu_{\bx}\vv v\|=\|\vv v\|\ge1$ for all $\bx$.
Let $1\le j\le d$, $\vv v=\vv v_1\wedge\dots\wedge\vv v_j\in\bigwedge^j\left(\Z^{^{d+1}}\right)$ and $\vv v\neq\vv0$. Let $\tbx=(1,x_1,\dots,x_d)$ and write each $\vv v_i=\left(v_{i,+},v_{i,1},\dots,v_{i,d}\right)$. Then
\begin{equation}\label{vb767}
gu_{\bx}\vv v=\bigwedge_{i=1}^j
\left(\begin{array}{c}
b_+\langle\tbx,\vv v_i\rangle\\
b_1v_{i,1}\\
\vdots\\
b_dv_{i,d}
\end{array}\right).
\end{equation}
Since $\vv v\neq\vv0$, there exists a collection $\{\ell_2,\dots,\ell_d\}\subq \{1,\dots,d\}$ such that the rows $\left(b_{\ell_k}v_{1,\ell_k}, \dots , b_{\ell_k}v_{j,\ell_k}\right)$ $(2\le k\le d)$ are linearly independent. It follows that the determinant
$$
\det\left(\begin{array}{ccc}
b_+\langle\tbx,\vv v_1\rangle &\dots & b_+\langle\tbx,\vv v_j\rangle\\
b_{\ell_2}v_{1,\ell_2}& \dots & b_{\ell_2}v_{j,\ell_2}\\
\vdots & \ddots &\vdots\\
b_{\ell_j}v_{1,\ell_j}&\dots &b_{\ell_j}v_{j,\ell_j}
\end{array}\right)=
\prod_{k=1}^jb_{\ell_k}\times\det\left(\begin{array}{ccc}
\langle\tbx,\vv v_1\rangle &\dots & \langle\tbx,\vv v_j\rangle\\
v_{1,\ell_2}& \dots & v_{j,\ell_2}\\
\vdots & \ddots &\vdots\\
v_{1,\ell_j}&\dots &v_{j,\ell_j}
\end{array}\right),
$$
where $\ell_1=+$, is not identically zero. Here we used the obvious fact that the functions $\langle\tbx,\vv v_1\rangle,\dots,\langle\tbx,\vv v_j\rangle$ are linearly independent over $\R$, which follows from the linear independence of $\vv v_1,\dots,\vv v_j$. Observe that the above determinant is one of the coordinates of $gu_{\bx}\vv v$ (see \cite[Section IV.6 Lemma 6A]{Schmidt3}). Furthermore, since all the vectors $\vv v_i$ are integer, it is of the form $\prod_{k=1}^jb_{\ell_k}f(\bx)$, where  $f(\bx)=c_0+c_1x_1+\dots+c_dx_d$ for some integer coefficients $c_0,\dots,c_d$, not all zeros. Since the norm of $gu_{\bx}\vv v$ is at least the absolute value of any of its coordinates, using the assumptions that $b_+\ge 1$, $0 < b_1,\ldots,b_d \leq 1$ and $b_+b_1\cdots b_d=1$ gives
\begin{equation}\label{vb777}
\|gu_{\bx}\vv v\|\ge \left|\prod_{k=1}^jb_{\ell_k}f(\bx)\right|\ge
\left|f(\bx)\right|.
\end{equation}
If $c_1=\dots=c_d=0$, then the right hand side of \eqref{vb777} is a nonzero integer and therefore it is at least 1. Otherwise, $c_k\neq0$ for some $1\le k\le d$. Then, take the points $\bx_{\pm1}=\bx_0\pm\,r_B\vv e_k$, where $\bx_0$ is the centre of $B$. Then, $|f\left(\bx_{+1}\right)-f\left(\bx_{-1}\right)|=|2c_kr_B|\ge2\,r_B$. Consequently, using the triangle inequality, we get that $\sup_{\bx\in B}\|gu_{\bx}\vv v\|\ge \max\left\{\left|f\left(\bx_{+1}\right)\right|,\left|f\left(\bx_{-1}\right)\right|\right\}\ge r_B$. The proof is complete.
\end{proof}

\section{Proof of Theorem~\ref{thm:intersection}}\label{sec:proof}

Let $\bw$ be any weight, $\mu$ be a compactly supported Ahlfors regular absolutely decaying measure on $\R^d$ and let $A,\alpha$ and $\rho_0$ be as in \eqref{eq:AhlforsRegular}. For every $\bz\in\supp\mu$ let $B(\bz)$ be the ball arising from Theorem~\ref{thm:KLW}. Clearly, $\left\{\frac12B(\bz):\bz\in\supp\mu\right\}$ is an open cover of $\supp\mu$. Since $\supp\mu$ is compact, there is a finite subcover $\left\{\frac12B\left(\bz_u\right): 1\le u\le U\right\}$ of $\supp\mu$. Thus,
\begin{equation}\label{cover}
\supp\mu\subq\bigcup_{u=1}^U \tfrac12B\left(\bz_u\right).
\end{equation}

\begin{prop}\label{C-gamma}
There exist positive constants $C$ and $\gamma$ with the following property. For every ball $B$ centred in $\supp\mu$ that is contained in one of the balls $B\left(\bz_u\right)$ with $1\le u\le U$ and such that statement (ii) of Theorem~\ref{thm:KLW} does not hold
inequality \eqref{eq:KLW} holds for all $\eps>0$.
\end{prop}

\begin{proof}
The existence of $C$ and $\gamma$ follows from Theorem~\ref{thm:KLW} since we have a finite collection of balls $B\left(\bz_u\right)$ and so $C$ can be taken its maximal value over $B\left(\bz_u\right)$ and $\gamma$ can be taken its minimal value over $B\left(\bz_u\right)$.
\end{proof}

Recall that the ultimate goal is to show that $\bad(\bw)\cap\support\mu\neq\varnothing$.
Note that the support of a Borel measure is closed thus $X := \support \mu\subq\bbr^d$ is complete. Then, by Theorem~\ref{thm:cantor-winning-nonempty} with $S :=\badw \cap \support \mu$, it suffices to show that $S$ is $\alpha'$-Cantor winning (in the sense of the Cantor potential game played on $X$) for
some $0\leq\alpha'<\alpha$. The specific value of $\alpha'$ we use will be defined in \eqref{eq:alphatag} below.

We will describe a winning strategy for Alice for the $\alpha'$-Cantor potential game. Assume Bob chooses $B_0$ and $\beta$ on his first move. Recall that $B_0$ is a closed ball in $X=\supp\mu$ defined by its centre $\bx_0{\in X}$ and radius $r_0$, {that is $B_0=X\cap B(\bx_0,r_0)$}. Before describing Alice's strategy we start with several simplifying assumptions. Without loss of generality we can assume that $r_0$ is less than $\tfrac{1}{10}$ of the radius of every ball $\frac12B\left(\bz_u\right)$  $(1\le u\le U)$ appearing in \eqref{cover}. This can be done as a result of Alice playing arbitrarily for several moves until the condition is met.
Let $u_0$ be such that $\bx_0\in\tfrac12B\left(\bz_{u_0}\right)$, {which exists due to \eqref{cover}.} Then using the triangle inequality and the above condition on $r_0$ we conclude that
\begin{equation}\label{inclusion}
B(\bx_0,5r_0)\subq B(\bz),\qquad\text{where $\bz=\bz_{u_0}$}\,.
\end{equation}
Here $B(\bx_0,r_0)$ is the ball in $\R^d$ of radius $r_0$ centred at $\bx_0$.
Also without loss of generality we will assume that
$$
r_0\leq \min\left\{\tfrac13\rho_0,\frac{1}{\sqrt{d}},\;A^{-1/\alpha}\right\},
$$
where $\rho_0$, $A$ and $\alpha$ are as in \eqref{eq:AhlforsRegular}.
In particular, by \eqref{eq:AhlforsRegular}, we have that
\begin{equation}\label{eq:measureOfB0}
\mu(B_0)\le 1\,.
\end{equation}
Without loss of generality we will assume that
\[
w_1\ge w_2\ge \ldots\geq w_d > 0\,.
\]
Define formally $w_{d+1}:=0$. Then there exists a unique integer $t$ such that $1\leq t\leq d$ and
\[
w_1=\ldots=w_t > w_{t+1}\,.
\]
Without loss of generality we may assume that $\beta$ is small (to be determined according to \eqref{eq:largeEnough1} and \eqref{eq:largeEnough2}). 
This can be done as a result of applying Alice's strategy described below with $\beta^M$ for some integer $M\geq1$ and 
letting Alice play arbitrarily on every step of the game which is not $1$ modulo $M$. Let $b>1$ be such that
\[
\beta = b^{-(1+w_1)}\,.
\]
Recall that $b$ is a parameter appearing in the definition of $a_n$, see \eqref{eq:an}. For any $n\geq0$ denote Bob's $n$th move by
\[
B_n=X\cap B\left(\bx_n,\beta^nr_0\right).
\]
Thus, $B_n$ is a ball in $X$ of radius $\beta^nr_0$ centred at $\bx_n\in X$. For every integer $\ell\in\bbz$ define {the following diagonal matrix}
{\begin{equation}\label{eq:dl}
d_{\ell}:=\diag\left(\beta^{\frac{t\ell}{d+1}},~\underbrace{\beta^{-\frac{(d+1-t)\ell}{d+1}},\dots, \beta^{-\frac{(d+1-t)\ell}{d+1}}}_{\text{$t$ times}},
 ~\underbrace{\beta^{\frac{t\ell}{d+1}},\dots,\beta^{\frac{t\ell}{d+1}}}_{\text{$d-t$ times}}\right) \in G\,.
\end{equation}
Also, given any $\eps > 0$, $B\subq \R^d$ and non-negative integers $n$, $k$ and $\ell$, let
\begin{equation}\label{key_sets}
A^{k,\ell}_\eps(B) :=\left\{ \bx\in B \sep d_{\ell}a_{k}u_{\bx}\notin K_{\eps}\right\}.
\end{equation}
When $B=B_n$ we will write $A^{k,\ell,n}_\eps$ for $A^{k,\ell}_\eps(B_n)$.
The sets $A^{k,\ell,n}_\eps$ will be used to define Alice's winning strategy, see \eqref{eq:extremelyDangerous}--\eqref{thestrategy}. 
It will be apparent from the definition of Alice's strategy and Definition~\ref{def:cantorWinning} that our proof crucially 
depends on obtaining suitably precise upper bounds on the $\mu$-measure of certain neighborhoods of the sets $A^{k,\ell,n}_\eps$. 
The following lemma, which provides such bounds, is therefore the key step in defining Alice's winning strategy.

\begin{lem}\label{lem:keyLemma}
Let $C$ and $\gamma$ be the same as in Proposition~\ref{C-gamma} and let
$$
C'=C A^2\max\left\{2^\alpha(2r_0)^{-\gamma},\, 3^\alpha \right\}.
$$
Then for any quintuple of nonnegative integers $(h,k,\ell,m,n)$, if
\begin{align}
d_{\ell+m}a_{k}u_{\zblue{\tilde\bx_n}}\bbz^{d+1} & \in K_{\sqrt{d+1}\beta^{\frac{m}{d+1}}{r_0^{\zblue{-1}}}}\,,\label{eq:firstAssumption}
\end{align}
and
\begin{align}
d_\ell a_{k-h}u_{\zblue{\tilde\bx_n}}\bbz^{d+1} & \in K_{\zblue{\sqrt{2}d}\beta^{\frac{\tau}{d}}},\label{eq:secondAssumption}
\end{align}
for some point $\tilde\bx_n\in B_n$, where
\begin{equation}\label{tau}
\tau=\zblue{\min}\left\{k-\ell-m-n-\frac{h}{1+w_1},\frac{hw_d}{1+w_1}\right\}\zblue{\ge0}\,,
\end{equation}
then for any $\eps > 0$
\begin{equation}\label{eq:result1}
\mu\left(A^{k,\ell,n}_\eps\right) \leq \zblue{C'}\eps^{\gamma}\mu\left(B_n\right).
\end{equation}
If $n=0$ and $k\geq \frac{1+w_1}{w_1}\ell$ then \eqref{eq:result1}
%\begin{equation}\label{eq:result11}
%\mu\left(A^{k,\ell,0}_\eps\right)
%\leq C' \eps^{\gamma}\mu\left(B_{\blue{0}}\right)
%\end{equation}
holds for any $\eps>0$ without assuming \eqref{eq:firstAssumption} and \eqref{eq:secondAssumption}.
%\noindent
Moreover, if $0<r\le \beta^{n}r_0$ and
\begin{equation}\label{eq:rho}
\zblue{\eps' := \left(1+\max\left\{\beta^{\ell-k},b^{\left(1+w_{t+1}\right)k}\right\}r\right)\eps \,,}
\end{equation}
where $w_{t+1}=0$ if $t=d$, then
\begin{equation}\label{eq:result2}
\mu\left(B\left(A^{k,\ell,n}_\eps,r\right) \right) \leq C'{\eps'}^{\gamma}\mu\left(B_n\right),
\end{equation}
\zblue{which is valid for all $n$ assuming \eqref{eq:firstAssumption} and \eqref{eq:secondAssumption} and for $n=0$ assuming $k\geq \frac{1+w_1}{w_1}\ell$.}
\end{lem}

\begin{proof}
Let $n\ge0$ and write $\bx=\tilde\bx_n+\bx'$ with $\left\|\bx'\right\| \le \beta^{n}r_0$. Since $w_1=\dots=w_t$ and $\beta=b^{-(1+w_1)}$,
conjugating $u_{\bx'}$ by $d_\ell a_k$ in the equation $u_{\bx}=u_{\bx'}u_{\tilde\bx_n}$ gives
\begin{equation}\label{eq:dlakconj}
d_\ell a_{k}u_{\bx} =
u_{\left(\beta^{\ell-k}x_1',\ldots,\beta^{\ell-k}x_t',b^{\left(1+w_{t+1}\right)k}x_{t+1}',\ldots,b^{\left(1+w_d\right)k}x_{d}'\right)}d_\ell a_{k}u_{\tilde\bx_n}\,.
\end{equation}
Note that \eqref{inclusion} and the fact $\tilde\bx_n\in B_n=X\cap B\left(\bx_n,\beta^nr_n\right)$ imply that
\begin{equation}\label{vb2000}
B\left(\bx_n, 2\beta^nr_0\right)\subq B\left(\tilde\bx_n, 3\beta^nr_0\right)\subq B\left(\bx_n, 4\beta^nr_0\right)
\subq B\left(\bx_0, 5r_0\right)\subq B(\bz)\,,
\end{equation}
\zblue{where $B(\bz)$ is as in \eqref{inclusion}}. Let $1\leq j\leq d$ and $0\neq\bv=\bv_1\wedge\ldots\wedge\bv_j\in\bigwedge^j\left(\bbz^{d+1}\right)$.
Let $\bv'=d_\ell a_{k}u_{\zblue{\tilde\bx_n}}\bv$ and
$\bv_i'=d_\ell a_{k}u_{\zblue{\tilde\bx_n}}\bv_i$ for every $1\leq i\leq j$. Assume towards a contradiction that
\begin{equation}\label{eq:contrary}
\sup_{\|\bx'\|\le\beta^{n}{r_0}}\left\|u_{\left(\beta^{\ell-k}x_1',\ldots,\beta^{\ell-k}x_t',b^{\left(1+w_{t+1}\right)k}x_{t+1}',\ldots,b^{\left(1+w_d\right)k}x_{d}'\right)}\bv'\right\|<1\,.
\end{equation}

Inequality \eqref{eq:contrary} applied at $\bx'=0$ implies that $\left\|\bv'\right\|<1$. Since $\|\,\,\|$ is the Euclidean norm, $\left\|\bv'\right\|=\left\|\bv'_1\wedge\ldots\wedge\bv'_j\right\|$ is the covolume of the lattice $\Lambda^*:=\Z\bv'_1+\ldots+\Z\bv'_j$ in the Euclidean subspace $W^*:=\R\bv'_1+\ldots+\R\bv'_j$ of $\R^{d+1}$. That is $\operatorname{covol}\left(\Lambda^*\right)<1$. Next, as is well known the Euclidean ball $B^*_{\sqrt j}$ in $W^*$ centred at $0$ of radius $\sqrt j$ contains a cube $\mathcal{C}^*$ of sidelength $2$, {\em i.e.}\/ $\mathcal{C}^*:=\left\{\sum_{i=1}^j\theta_i\be^*_i\sep |\theta_i|\le1\right\}$, where $\be^*_1,\dots,\be^*_j$ is any orthonormal basis in $W^*$. Therefore $B^*_{\sqrt j}$ has $j$-dimensional volume $> 2^j$. Then, by Minkowski's theorem for convex bodies, $B^*_{\sqrt j}$ contains a non-zero point of $\Lambda^*$. In other words, the shortest non-zero vector of $\Lambda^*$, say $\tilde\bv'_1$, has Euclidean norm $\le\sqrt j$. Complete $\tilde\bv'_1$ to a basis $\tilde\bv'_1,\dots,\tilde\bv'_j$ of $\Lambda^*$, {\em e.g.}\/ to a reduced Minkowski basis. Then $\bv'_1\wedge\ldots\wedge\bv'_j=\pm\tilde\bv'_1\wedge\ldots\wedge\tilde\bv'_j$, since the two bases span the same linear subspace of $\R^{d+1}$ and have the same Euclidean norm (equal to the covolume of $\Lambda^*$). To save on notation, without loss of generality we will assume that $\bv'_i=\tilde\bv'_i$ $(1\le i\le j)$. Then, we have that
\begin{equation}\label{eq:minkowski}
\left\|\bv_1'\right\|<\sqrt j\le \sqrt{d}  \,\le \zblue{r_0^{-1}}\,.
\end{equation}

If $\left|v_{1,i}'\right|<\beta^{m}r_0^{\zblue{-1}}$ for all $1\leq i\leq t$, where $\bv'_1=\left(v_{1,+}',v_{1,1}',\dots,v_{1,d}'\right)$, then using \eqref{eq:minkowski} we get that
\begin{equation}\label{eq:rem}
\|d_{m}\bv_1'\| < \zblue{\sqrt{d+1}\beta^{\frac{t m}{d+1}}r_0^{-1}\le }\sqrt{d+1}\beta^{\frac{m}{d+1}}r_0^{\zblue{-1}}\,,
\end{equation}
and so, \zblue{since $\bv_1'=d_\ell a_{k}u_{\zblue{\tilde\bx_n}}\bv_1$ with $\bv_1\in\Z^{d+1}\setminus\{0\}$ and $d_md_\ell=d_{\ell+m}$, we get that}
$$
d_{\ell+m}a_{k}u_{\zblue{\tilde\bx_n}}\bbz^{d+1}\notin K_{\sqrt{d+1}\beta^{\frac{m}{d+1}} r_0^{\zblue{-1}}}\,,
$$
which contradicts assumption \eqref{eq:firstAssumption}.

Otherwise, there exists $1\leq i_0 \leq t$ for which
\begin{equation}\label{eq:lemma_v11}
\left|v_{1,i_0}'\right|\geq \beta^{m}r_0^{\zblue{-1}}\,.
\end{equation}
It is enough to use \eqref{eq:contrary} for $\bx'$ of the form $\bx'=\left(0,\ldots,x_{i_0}',\ldots,0\right)$
where the only nonzero entry is in the $i_0$th coordinate. In this case, let
\begin{equation}\label{eq:tilde}
\tilde{\bv}=\sum_{i=1}^j (-1)^{i+1}v_{i,i_0}'\bigwedge_{i'\neq i}\bv_{i'}'\,.
\end{equation}
Then, {by \eqref{eq:uaction}, we have that}
\begin{equation}\label{eq:lemma_uaction_xaxis}
u_{\left(\beta^{\ell-k}x_1',\ldots,\beta^{\ell-k}x_t',b^{\left(1+w_{t+1}\right)k}x_{t+1}',\ldots,b^{\left(1+w_d\right)k}x_{d}'\right)}\bv'=\bv' + \beta^{\ell-k}x_{i_0}'\be_{+}\wedge\tilde{\bv}
\end{equation}
\zblue{for all $\bx'=\left(0,\ldots,x_{i_0}',\ldots,0\right)$ with $|x_{i_0}'|\le\beta^{n}{r_0}$. By \eqref{eq:contrary},
$
\left\|\bv' \pm \beta^{\ell-k}x_{i_0}'\be_{+}\wedge\tilde{\bv}\right\|<1
$
when $x_{i_0}'=\beta^{n}r_0$. Then, by the triangle inequality, $$
\left\|2\beta^{\ell-k}x_{i_0}'\be_{+}\wedge\tilde{\bv}\right\| \le \left\|\bv' + \beta^{\ell-k}x_{i_0}'\be_{+}\wedge\tilde{\bv}\right\| + \left\|\bv' - \beta^{\ell-k}x_{i_0}'\be_{+}\wedge\tilde{\bv}\right\|<2\,.
$$
Letting $x_{i_0}'=\beta^{n}r_0$ and dividing the above
inequality through by $2\beta^{\ell-k}x_{i_0}'$ give}
\begin{equation}\label{eq:lemma_vector}
\left\|\be_+\wedge\tilde{\bv}\right\| < \beta^{k-\ell-n}\zblue{r_0^{-1}}\,.
\end{equation}
\zblue{Observe that $\be_+\wedge\tilde{\bv}_+=0$, and thus
$\be_+\wedge\tilde{\bv}=\be_+\wedge\tilde{\bv}_-$. Furthermore,
$\left\|\be_+\wedge\tilde{\bv}_-\right\|=\|\tilde{\bv}_{-}\|$} and therefore \eqref{eq:lemma_vector} can be rewritten as
\begin{equation}\label{eq:lemma_vector1}
\left\|\tilde{\bv}_{-}\right\| < \beta^{k-\ell-n}\zblue{r_0^{-1}}\,.
\end{equation}
On the other hand, by \eqref{eq:tilde}, taking the wedge product of $\bv_{1}'$ and $\tilde{\bv}$ gives
\begin{equation}\label{eq:lemma_wedge}
\zblue{\bv_{1}'\wedge\tilde{\bv} = v_{1,i_0}'\bv'\,,}
\end{equation}
so equations \eqref{eq:minkowski}, \eqref{eq:lemma_v11}, \eqref{eq:lemma_vector1} and \eqref{eq:lemma_wedge} yield
\[
\left\|\bv_{-}'\right\| \leq \left|v_{1,i_0}'\right|^{-1}\left\|\bv_{1,-}'\right\|\left\|\tilde{\bv}_{-}\right\| < \zblue{\sqrt{d}}\beta^{k-\ell-m-n}\,.
\]
Applying the right hand side of \eqref{eq:aaction} gives
\begin{equation}\label{vb701}
\left\|a_{-h}\bv_{-}'\right\| < \zblue{\sqrt{d}}\beta^{k - \ell - m - n - \frac{h}{1+w_1}}\,.
\end{equation}
\zblue{Next, by \eqref{eq:contrary}, $\|\bv_{+}'\|\le\|\bv'\|<1$, and so applying
the left hand side of \eqref{eq:aaction} gives}
\begin{equation}\label{vb702}
\left\|a_{-h}\bv_{+}'\right\| < b^{-w_dh}\,.
\end{equation}
\zblue{Combining \eqref{vb701} and \eqref{vb702} together with the fact
that $a_{-h}\bv'=d_\ell a_{k-h}u_{\zblue{\tilde\bx_n}}\bv$} we obtain
\begin{equation}\label{eq:this}
\left\|d_\ell a_{k-h}u_{\zblue{\tilde\bx_n}}\bv\right\| = \left\|a_{-h}\bv'\right\| < \sqrt{2\zblue{d}}\max\left\{\beta^{k - \ell - m - n - \frac{h}{1+w_1}},b^{-hw_d}\right\}\zblue{=\sqrt{2d}\beta^{\tau}\,,}
\end{equation}
\zblue{where $\tau$ is given by \eqref{tau}.}
Using Minkowski's theorem for convex bodies in the same way as we did in the argument leading to \eqref{eq:minkowski},
we deduce from \eqref{eq:this} that the lattice $d_\ell a_{k-h}u_{}\bbz^{d+1}$ has a nonzero vector which Euclidean norm is smaller 
than
$$
\sqrt{j}\left(\sqrt{2d}\beta^\tau\right)^\frac{1}{j}\le
\sqrt{2}d\beta^\frac{\tau}{j}\le \sqrt{2}d\beta^\frac{\tau}{d}
$$
since $\tau\ge0$ and $0<\beta<1$, contrary to \eqref{eq:secondAssumption}. Thus, \eqref{eq:contrary} cannot hold and therefore, 
by \eqref{eq:dlakconj}, we have that Condition (ii) within Theorem~\ref{thm:KLW} cannot hold with $\rho=1$ and $B=B\left(\tilde\bx_n, 3\beta^nr_0\right)$. 
Hence, by Theorem~\ref{thm:KLW} with this choice of $\rho$ and $B$, which is applicable in view of Proposition~\ref{C-gamma} and \eqref{vb2000}, we obtain that
\begin{equation}\label{eq:result1+}
\mu\left(\left\{ \bx\in B\left(\tilde\bx_n,3\beta^nr_0\right) \sep d_{\ell}a_{k}u_{\bx}\notin K_{\eps}\right\}\right)\le C\eps^{\gamma}\mu\left(B(\tilde\bx_n,3\beta^nr_0)\right).
\end{equation}
Since $\bx_n$ and $\tilde\bx_n$ are both in the support of $\mu$ and $3r_0\le\rho_0$, by \eqref{eq:AhlforsRegular}, we have that
\begin{align}
\mu\left(B(\tilde\bx_n,3\beta^nr_0)\right)&\le A \left(3\beta^nr_0\right)^\alpha= 3^\alpha A^2 A^{-1} \left(\beta^nr_0\right)^\alpha\label{vb2003}\\
&\le 3^\alpha A^2\mu\left(B(\bx_n,\beta^nr_0)\right)= 3^\alpha A^2\mu(B_n)\,. \nonumber
\end{align}
Further observe that, by \eqref{vb2000}, the left hand side of \eqref{eq:result1+} is an upper bound for $\mu\left(A^{k,\ell}_\eps(2B_n)\right)$. Hence, combining \eqref{eq:result1+} and \eqref{vb2003} gives
\begin{equation}\label{eq:result1++}
\mu\left(A^{k,\ell}_\eps(2B_n)\right) \leq C'\eps^{\gamma}\mu\left(B_n\right).
\end{equation}
And since trivially we have that $A^{k,\ell,n}_\eps\subq A^{k,\ell}_\eps(2B_n)$, \eqref{eq:result1++} 
implies \eqref{eq:result1}, as required.

Regarding the case $n=0$ first observe that, by \eqref{eq:AhlforsRegular}, we have that
\begin{equation}\label{vb4000}
\mu(2B_n)\le A^22^\alpha\mu(B_n).
\end{equation}
Now if $k\geq \frac{1+w_1}{w_1}\ell$ then $d_\ell a_k = \diag(b_+,b_1,\dots,b_d)$
with $b_{+}\geq1$ and $b_i\leq1$ for every $1\leq i\leq d$, so Theorem~\ref{thm:KLW} with $\rho=2r_0$ together
with Lemma~\ref{thm:KLW3.3} and \eqref{vb4000} immediately imply \eqref{eq:result1++} and consequently \eqref{eq:result1}.

To see \eqref{eq:result2}, assume that $0<r\le\beta^nr_0$ and $\eps'$ is given by \eqref{eq:rho}. If $\bx \in A^{k,\ell,n}_\eps$
then there exists $\bv\in\bbz^{d+1}\setminus\{0\}$ such that $\|d_\ell a_k u_\bx \bv \|<\eps$. Suppose that $\|\by-\bx\| < r$ and
define $\bx'=\by-\bx$. Using the conjugation of $u_{\bx'}$ by $d_\ell a_k$ as in \eqref{eq:dlakconj} we get that
\begin{align*}
d_\ell a_{k}u_{\by}\bv & = u_{\left(\beta^{\ell-k}x_1',\ldots,\beta^{\ell-k}x_t',b^{\left(1+w_{t+1}\right)k}x_{t+1}',\ldots, b^{\left(1+w_d\right)k}x_{d}'\right)}d_\ell a_{k}u_{\bx}\bv\\[1ex]
&= d_\ell a_{k}u_{\bx}\bv+\langle\tilde\bv,d_\ell a_{k}u_{\bx}\bv\rangle\be_+\,,
\end{align*}
where
$
\tilde\bv=\left(0,\beta^{\ell-k}x_1',\ldots,\beta^{\ell-k}x_t',b^{\left(1+w_{t+1}\right)k}x_{t+1}',\ldots, b^{\left(1+w_d\right)k}x_{d}'\right).
$
Then on using the triangle and Cauchy-Schwarz inequalities we get that
$$
\left\|d_\ell a_{k}u_{\by}\bv\right\| \le
\left\|d_\ell a_{k}u_{\bx}\bv\right\|+|\langle\tilde\bv,d_\ell a_{k}u_{\bx}\bv\rangle|\le
\left(1+\|\tilde\bv\|\right)\left\|d_\ell a_{k}u_{\bx}\bv\right\|.
$$
Observe that
$$
\left\|\tilde\bv\right\|\le \max\left\{\beta^{\ell-k},b^{\left(1+w_{t+1}\right)k}\right\}\cdot\|\bx'\|
\le \max\left\{\beta^{\ell-k},b^{\left(1+w_{t+1}\right)k}\right\}r
$$
and therefore
\[
\|d_\ell a_k u_{\by} \bv\| < \left(1+\max\left\{\beta^{\ell-k},b^{\left(1+w_{t+1}\right)k}\right\}r\right)\eps = \eps'\,.
\]
Further, since $\bx \in A^{k,\ell,n}_\eps\subq B_n$ and $\|\by-\bx\|\le r\le \beta^n r_0$, we have that $\by\in2B_n$.
Then, $B\left(A^{k,\ell,n}_\eps,r\right)\subq  A^{k,\ell}_{\eps'}(2B_n)$ and applying \eqref{eq:result1++} with $\eps$ replaced by $\eps'$ gives \eqref{eq:result2}.
\end{proof}

\medskip

To complete the proof of Theorem~\ref{thm:intersection}, let
\begin{align}
s& :=\max\left\{ 5, \left\lceil\frac{1+w_1}{w_1-w_{t+1}}\right\rceil,\left\lceil\frac{2(1+w_1)+1}{w_1}\right\rceil\right\}\label{eq:s},\\
\eta    & := \min\left\{\frac{1}{4(d+1)},\frac{w_ds}{2d(1+w_1)},\frac{w_1s-2(1+w_1)}{2d(1+w_1)},\frac{\alpha}{\gamma}\right\}\label{eq:eta},\\
\alpha' & := \alpha-\frac{\gamma\eta}{4(s-1)}\label{eq:alphatag}\,,
\end{align}
where it is agreed, if needed (i.e., in case $t=d$), that $w_{d+1}=0$. Note that \eqref{eq:s} and \eqref{eq:eta} imply that $\eta>0$, $0\leq\alpha'<\alpha$ and $s\geq5$. Assume that $\beta$ is small enough so that it satisfies
\begin{align}
&\beta^{-\frac{w_1s-2(1+w_1)}{2d(1+w_1)}} \geq \zblue{\sqrt{2}d},\quad \beta^{-\frac{w_ds}{2d(1+w_1)}} \geq \zblue{\sqrt{2}d} \quad\text{and}\quad \beta^{-\frac{1}{\zblue{2}(d+1)}} \geq \sqrt{d+1}\zblue{r_0^{-1}}\,,\label{eq:largeEnough1}\\
&\beta^{\frac{\gamma\eta}{s-1}} \leq \min\left\{2^{-1} ,\zblue{\left(A^2C'2^{\gamma+1}(3r_0^{-1})^\alpha\right)^{-1}}  \right\}.\label{eq:largeEnough2}
\end{align}
\zblue{Recall that $A$ and $\alpha$ are the parameters of $\mu$ appearing in \eqref{eq:AhlforsRegular}, $C'$ is given in Lemma~\ref{lem:keyLemma}, $C$ and $\gamma$ are as in Proposition~\ref{C-gamma} and $r_0$ is the radius of $B_0$ which choice is described after Proposition~\ref{C-gamma}.}

Now let us describe the winning strategy for Alice. We will keep notation used in Definition~\ref{def:cantorWinning}.
Alice's strategy is understood as a sequence of maps $\cF_n$ indexed by $n\in\N$ which assign a 
sequence $\left(\mathcal{A}_{n+1,i}\right)_{i\ge0}$ of collections of at 
most $\beta^{-\alpha(i+1)}$ balls of radius $\beta^{n+1+i}r_0$ to Bob's moves $\beta,B_0,\dots,B_n$, that is $\left(\mathcal{A}_{n+1,i}\right)_{i\ge0}=\cF_n(\beta,B_0,\dots,B_n)$. The sets $\mathcal{A}_{n+1,i}$ will be defined in a 3-step process, which can be described as follows:
\begin{itemize}
  \item[(i)] Choose `raw' subsets $A_{n+1, i} \subq B_n$ that Alice wishes to `block out' as she plays the game; then
  \item[(ii)] Refine $A_{n+1, i}$ to obtain subsets $\tilde A_{n+1, i}\subq A_{n+1, i}$ by removing overlaps between the sets $A_{n+1, i}$ that appear at different stages of the game; and
  \item[(iii)] Finally `convert' $\tilde{A}_{n+1,i}$ into the required collections of balls $\mathcal{A}_{n+1,i}$ by using the efficient covering of Lemma~\ref{lem:covering}.
\end{itemize}

Naturally we start with Step~(i) to define the raw sets $A_{n+1, i}$. To begin with, for every $i\geq0$, let
\begin{equation}\label{eq:extremelyDangerous}
A_{1,i} :=  \bigcup_{\substack{\zblue{n'}\geq0,\;\ell \geq \frac{\zblue{n'}+1}{s}\\[0ex] \zblue{n'} + (s-1)\ell = i}} \left\{ \bx \in B_0\sep d_{\ell} a_{\zblue{n'}+1+s\ell}u_{\bx}\bbz^{d+1}\notin K_{\beta^{\eta\ell}}\right\},
\end{equation}
where the union is taken over all possible values of integers $\ell$ and $n'\ge0$.
Note that if no such values of $\ell$ and $n'$ exist, then $A_{1,i}=\varnothing$. Next,
for $n\geq 1$ define
\begin{equation}\label{eq:dangerous}
A_{n+1,i} := \left\{ \bx\in B_{n}\sep d_{\ell}a_{n+1+s\ell}u_{\bx}\bbz^{d+1}\notin K_{\beta^{\eta \ell}}\right\}
\end{equation}
if $i=(s-1)\ell$ for some integer $0<\ell<\frac{n+1}{s}$, and we define $A_{n+1,i}:=\varnothing$ otherwise, that is for $i\ge0$ such that $i\notin\left\{(s-1)\ell\sep \ell\in\Z,\;0<\ell<\frac{n+1}{s}\right\}$.

Now moving on to Step~(ii), define
\begin{equation}\label{eq:dangerous2}
\tilde A_{n+1,i} = A_{n+1,i}\setminus \bigcup_{0\le n'<n,\; i'\ge0} A_{n'+1, i'}\,,
\end{equation}
where in the case $n=0$ the union \eqref{eq:dangerous2} is empty and thus $\tilde A_{1,i} = A_{1,i}$.

Finally, moving on to Step~(iii), with reference to Lemma~\ref{lem:covering},
\begin{equation}\label{thestrategy}
\text{let $\cala_{n+1,i}$ be an efficient cover of $\tilde A_{n+1,i}$ by balls of radius $\beta^{n+1+i}r_0$.}
\end{equation}
Since $\cala_{n+1,i}$ is a cover of $\tilde A_{n+1,i}$, by Definition~\ref{def:cantorWinning}, when 
Bob makes his next move $B_{n+1}$ it must be disjoint from $\tilde A_{n'+1,i'}$ for all $i',n'\ge0$ such that $n'+i'=n$. 
Therefore, if $\bx$ is an outcome of the game, that is $\bx\in \cap_{n\ge0}B_n$, then we necessarily have that
\begin{equation}\label{win!}
\bx\not\in \bigcup_{n,i\ge0}\tilde A_{n+1,i}\,.
\end{equation}
Using a standard inclusion-exclusion argument and \eqref{eq:dangerous2}, one readily verifies that
$$
\bigcup_{n,i\ge0}\tilde A_{n+1,i}=\bigcup_{n,i\ge0}A_{n+1,i}
$$
and therefore, by \eqref{win!},
\begin{equation}\label{win!!}
\bx\not\in \bigcup_{n,i\ge0}A_{n+1,i}\,.
\end{equation}
By \eqref{eq:dangerous2}, we have that the outcome $\bx$ of the game satisfies
\[
d_{\ell}a_{n+1+s\ell}u_{\bx}\bbz^{d+1}\in K_{\beta^{\eta\ell}}
\]
for every $n\geq0$ and $\ell\ge1$. Using this for $\ell=1$ we get that
\[
a_{n}u_{\bx}\bbz^{d+1}\in d_{-1}a_{-s-1}K_{\beta^{\eta}}\subq K_{\beta^{\eta+\frac{s+1}{1+w_1}+\frac{d+1-t}{d+1}}}
\]
for all $n\ge0$. By Dani's correspondence \eqref{eq:daniCorrespondence}, this means that $\bx\in\bad(\bw)$. In order to complete the proof that $\bad(\bw)$ is Cantor winning in $\support\mu$ it is left to show that Alice's strategy is legal, that is for all $n,i\geq 0$ we have that
\begin{equation}\label{eq:bound}
\#\cala_{n+1,i} \leq \beta^{-\alpha'(i+1)}\,.
\end{equation}

The plan is to use Lemma~\ref{lem:keyLemma} in order to get a measure estimate for small neighborhoods of the sets $\tilde A_{n+1,i}$ and 
then to apply Lemma~\ref{lem:covering} to derive \eqref{eq:bound}. \zblue{The definition of $\tilde A_{n+1,i}$ is designed to ensure that 
assumptions \eqref{eq:firstAssumption} and \eqref{eq:secondAssumption} hold when they are needed and so Lemma~\ref{lem:keyLemma} is 
applicable. Note that in the case $n=0$, which serves as the basis of the inductive argument, these assumptions are not needed. 
See Figure~\ref{figure} for a more geometrical description of the definition of Alice's strategy}.

\begin{figure}\label{figure}
\includegraphics[scale=0.75]{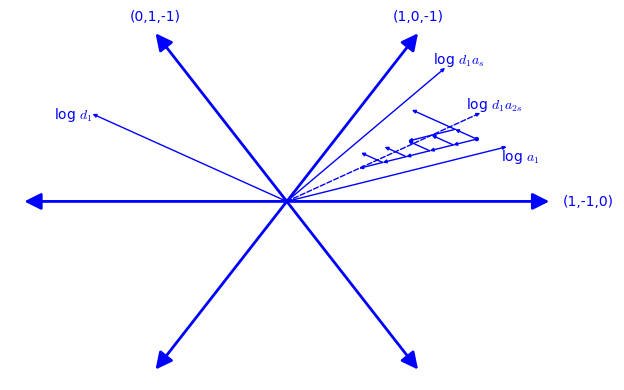}

\caption{\zblue{The plane represents all diagonal matrices in $\sltr$. Logarithm of a diagonal matrix is the vector whose coordinates are logarithms of the entries along the diagonal. Each blue point is the logarithm of a diagonal matrix $d_\ell a_{n+1+s\ell}$ for some parameters $n,\ell\geq0$. An arrow from $d_\ell a_{n+1+s\ell}$ to $d_{\ell'} a_{n'+1+s\ell'}$ is drawn if estimating the measure of $A_\eps^{n+1+s\ell,\ell,n}$ via applying Lemma~\ref{lem:keyLemma} requires an assumption regarding $d_{\ell'} a_{n'+1+s\ell'}u_{\bx_{n'}}$. Diagonal matrices in the region between the arrows which are labeled by $\log d_1 a_s$ and $\log d_1 a_{2s}$ are exactly those which satisfy $\ell\geq \frac{n+1}{s}$, and are dealt with at Alice's first turn. The parameters used to generate this figure are $\bw=(2/3,1/3)$, $s=2$, $n=5$, $\ell=1$.}}
\end{figure}

\zblue{First we deal with $A_{1,i}$.
Observe that the sets in the right hand side of \eqref{eq:extremelyDangerous} are precisely $A^{k,\ell,0}_\eps$ given by \eqref{key_sets}
with $k={n'}+1+s\ell$ and $\eps=\beta^{\eta\ell}$. Therefore,
\begin{equation}\label{vb1000}
B\left(A_{1,i},\threeo\tfrac13\beta^{i+1}\zblue{r_0}\right)=
\bigcup_{n'\geq0,\;\ell \geq \frac{n'+1}{s},\;n' + (s-1)\ell = i} B\left(A^{n'+1+s\ell,\ell,0}_{\beta^{\eta\ell}},\tfrac13\threeo\beta^{i+1}r_0\right).
\end{equation}}%
For any $i\geq0$, for any ${n'}\geq0$ and $\ell\geq\frac{{n'}+1}{s}$ such that $i = {n'} + (s-1)\ell$, apply Lemma~\ref{lem:keyLemma} with the quintuple $(h,k,\ell,m,n)$ set to be $(0,n'+1+s\ell,\ell,0,0)$, $\eps=\beta^{\eta\ell}$ and $r=\threeo\zblue{\tfrac13}\beta^{i+1}\zblue{r_0}$. In this case using \eqref{eq:rho}, \eqref{eq:s}, \zblue{the equation $b=\beta^{-1/(1+w_1)}$ and the fact that $0<\beta,r_0\le1$} gives
\begin{equation}\label{vb1002}
\eps'   = \left(1+\max\left\{\beta^{-(n'+1+(s-1)\ell)},b^{\left(1+w_{t+1}\right)(n'+1+s\ell)}\right\} \threeo\zblue{\tfrac13}\beta^{i+1}\zblue{r_0}\right)\beta^{\eta\ell} \le \threePlusOne\beta^{\eta\ell}\,.
\end{equation}
\zblue{By \eqref{eq:s} we have that $k\geq \frac{1+w_1}{w_1}\ell$ and thus Lemma~\ref{lem:keyLemma} is applicable to each set on the right of \eqref{vb1000}.} Therefore, using \zblue{\eqref{eq:result2}, \eqref{vb1000} and \eqref{eq:measureOfB0} gives}
\zblue{\begin{align}
\label{eq:third}
\mu\left(B\left(A_{1,i},\threeo\zblue{\tfrac13}\beta^{i+1}\zblue{r_0}\right)\right) & \leq \sum_{n'\geq0,\;\ell \geq \frac{{n'}+1}{s},\;{n'} + (s-1)\ell = i}C'{\threePlusOne^\gamma}\beta^{\gamma\eta\ell}\\
&{\leq \sum_{\ell\ge\frac{i+1}{2s-1}}C'\threePlusOne^\gamma \beta^{\gamma\eta\ell}
\leq \frac{C'\threePlusOne^\gamma}{1-\beta^{\gamma\eta}}\left(\beta^{\frac{\gamma\eta}{2s-1}}\right)^{i+1}}\nonumber\\
&{\leq C'\threePlusOne^{\gamma+1}\left(\beta^{\frac{\gamma\eta}{2s-1}}\right)^{i+1}}
{= C'\threePlusOne^{\gamma+1}\beta^{\frac{\gamma\eta}{2s-1}(i+1)}\beta^{\alpha n}\,,\nonumber}
\end{align}
since $n=0$, where the last inequality holds due to \eqref{eq:largeEnough2}.}

\zblue{Now let $n\ge1$, $i\ge0$ and let us assume without loss of generality that $\tilde A_{n+1,i}\neq\varnothing$ as otherwise $\cala_{n+1,i}=\varnothing$ and there is nothing to prove. In particular, we have that $i=(s-1)\ell\ge 4$ for some $1\le \ell < \frac{n+1}{s}$.
Observe that
$
A_{n+1,i} = A^{n+1+s\ell, \ell, n}_{\beta^{\eta \ell}}\,,
$
where $A^{n+1+s\ell, \ell, n}_{\beta^{\eta \ell}}$ is given by \eqref{key_sets}.
Therefore,
\begin{equation}\label{vb1001}
B\left(A_{n+1,i},\threeo\tfrac13\beta^{n+1+i}r_0\right) =
B\left(A^{n+1+s\ell,\ell,n}_{\beta^{\eta\ell}},\threeo\tfrac13\beta^{n+1+i}r_0\right).
\end{equation}
Note that $n+1+i=n+1+(s-1)\ell$. With the view to applying Lemma~\ref{lem:keyLemma} let the quintuple $(h,k,\ell,m,n)$ be $(s\ell,n+1+s\ell,\ell,\ell,n)$, $\eps={\beta^{\eta\ell}}$
and $r={\threeo\tfrac13\beta^{n+1+i}r_0=}\threeo\tfrac13\beta^{n+1+(s-1)\ell}{r_0}$. Using \eqref{eq:rho}, \eqref{eq:s}, {the equation $b=\beta^{-1/(1+w_1)}$ and the fact that $0<\beta,r_0\le1$}
in the same way as in \eqref{vb1002} we get that
\begin{align*}
\eps'  &= \left(1+\max\left\{\beta^{-(n+1+(s-1)\ell)},b^{\left(1+w_{t+1}\right)(n+1+s\ell)}\right\} \threeo{\tfrac13}\beta^{n+1+(s-1)\ell}{r_0}\right){\beta^{\eta\ell}}
  {\le} \threePlusOne{\beta^{\eta\ell}}\,.
\end{align*}
Further, since $\tilde A_{n+1,i}\neq\varnothing$, by \eqref{eq:dangerous2}, there exists a point $\tilde\bx_n\in B_n$ such that
\begin{equation}\label{vb3003}
\tilde\bx_n\not\in \bigcup_{0\le n'<n,\; i'\ge0} A_{n'+1, i'}\,.
\end{equation}
Recall that, by definition, $B_n$ is a subset of $\supp\mu$ and therefore $\tilde\bx_n\in\supp\mu$. By \eqref{vb3003},}
we have that
 \begin{align*}
     &\zblue{\tilde\bx_n\not\in} A_{1,n+(s-2)\ell}\cup A_{1, n-\ell}  & \text{ if } &  & \frac{n+1}{2s} \le  \ell < \frac{n+1}{s}\,,& \\
     &\zblue{\tilde\bx_n\not\in}  A_{1,n+(s-2)\ell} \cup A_{n+1-s\ell,(s-1)\ell}  & \text{ if } & & \frac{n+1}{3s}  \le  \ell < \frac{n+1}{2s}\,, &\\
     &\zblue{\tilde\bx_n\not\in} A_{n+1-s\ell,2(s-1)\ell}\cup A_{n+1-s\ell,(s-1)\ell}& \text{ if } & & \phantom{\frac{n+1}{3s} < {}} \ell < \frac{n+1}{3s}\,.&
 \end{align*}
A routine inspection of each of the sets above gives that
\begin{equation}\label{vb3001}
d_{2\ell}a_{n+1+s\ell}u_{\tilde\bx_n}\bbz^{d+1} 	 \in K_{\beta^{2\eta\ell}} \subq K_{\sqrt{d+1}\beta^{\frac{\ell}{d+1}}r_0^{{-1}}}\,,
\end{equation}
and
\begin{equation}\label{vb3002}
d_{\ell}a_{n+1}u_{\tilde\bx_n}\bbz^{d+1}  		 \in K_{\beta^{\eta\ell}} \subq K_{{\sqrt{2}d}\beta^{\frac{\tau}{d}}}\,,
\end{equation}
where
$$
\zblue{\textstyle
\tau=\min\left\{k-\ell-m-n-\frac{h}{1+w_1},\frac{hw_d}{1+w_1}\right\}=\min\left\{1 + (s-2)\ell - \frac{s\ell}{1+w_1},\frac{s\ell w_d}{1+w_1}\right\}>0\,.}
$$
and the containments on the right hand side follow from \eqref{eq:largeEnough1} together with \eqref{eq:eta}. Thus,  conditions \eqref{eq:firstAssumption} and \eqref{eq:secondAssumption} are satisfied. By \eqref{eq:result2} and \eqref{eq:AhlforsRegular} we have that
\zblue{
\[ \mu \left(  B\left(A^{n+1+s\ell,\ell,n}_{\beta^{\eta\ell}},\threeo\tfrac13\beta^{n+1+i}r_0\right) \right) \le C'2^\gamma \beta^{\gamma\eta \ell} \mu(B_n)\le C'2^\gamma \beta^{\gamma\eta \ell} A (\beta^nr_0)^\alpha.   \]
Then, by \eqref{vb1001} and the fact that $i=(s-1)\ell$, we get that
\begin{align}
\label{eq:forth}
\mu\left(B\left(A_{n+1, i},\tfrac{1}{3}\beta^{n+1+ i} r_0\right)\right) &\leq {C'2^\gamma r_0^\alpha A \beta^{\frac{\gamma\eta}{s-1} i}  \beta^{\alpha n}}\\
&=\nonumber
{C'2^\gamma r_0^\alpha A \beta^{\frac{\gamma\eta}{2(s-1)} (i+1)}  \beta^{\alpha n}}\,.
\end{align}
Combining \eqref{eq:third} and \eqref{eq:forth} together with Lemma~\ref{lem:covering}
applied with $r=\tfrac13\beta^{n+i+1}{r_0}$ gives that for every {$n\geq0, i \ge 4$}
\begin{align*}
\#\cala_{n+1,i}
& \leq \frac{A\mu\left(B\left(A_{n+1,i},\threeo\tfrac13\beta^{n+i+1}r_0\right)\right)}{\beta^{\alpha(n+i+1)}(\frac13r_0)^\alpha}\\[1ex]
& \le \frac{A\max\left\{C'\threePlusOne^{\gamma+1}\beta^{\frac{\gamma\eta}{2s-1}(i+1)}\beta^{\alpha n},\;
C'2^\gamma r_0^\alpha A \beta^{\frac{\gamma\eta}{2(s-1)} (i+1)}  \beta^{\alpha n}\right\}}{\beta^{\alpha(n+i+1)}(\frac13r_0)^\alpha}\\[1ex]
&\nonumber{\le A^2C'2^{\gamma+1}\left(3r_0^{-1}\right)^\alpha}\beta^{-(i+1)\left(\alpha-\frac{\gamma\eta}{2(s-1)}\right)}\\
&\nonumber{= A^2C'2^{\gamma+1}\left(3r_0^{-1}\right)^\alpha}\beta^{(i+1)\left(\frac{\gamma\eta}{4(s-1)}\right)}
\beta^{-(i+1)\left(\alpha-\frac{\gamma\eta}{4(s-1)}\right)}\\
& \leq \beta^{-\alpha'(i+1)}\,,
\end{align*}}%
where the last inequality follows from \eqref{eq:alphatag} and \eqref{eq:largeEnough2} and on using the fact that $i\ge4$.
This shows that the collection $\left\{\cala_{n+1,i}\sep n, i \in \bbn\right\}$ is a legal move for Alice.
By the argument in the beginning of this section, this completes the proof.

\begin{rem}
The diagonal matrices $d_\ell$ arise naturally in the proof of Lemma~\ref{lem:keyLemma}, even while only considering the sets $A_\eps^{k,\ell,n}$ with $\ell=0$. The general case is useful as the conditions \eqref{eq:firstAssumption} and \eqref{eq:secondAssumption} turn out to also be of the form $\bx_n\notin A_{\eps'}^{k',\ell',n'}$ for some parameters $\eps',k',\ell',n'$. Choosing $d_\ell$ as in \eqref{eq:dl} is natural, as it equally expands the $2$nd to $(t+1)$st coordinates which are all contracted by $a_k$ at the same rate. However, it is likely that $d_1$, say, can be replaced by any other unimodular diagonal matrix which expands the $2$nd to $(t+1)$st coordinates and contracts the other directions. This property is necessary, in order to ensure that an inequality similar to \eqref{eq:rem} holds. Changing the definition of $d_\ell$ in this fashion will require a different choice of the parameter $s$, which is defined in \eqref{eq:s}. In this context, it should be noted that Lemma~\ref{lem:keyLemma} is only applied with diagonal matrices of the form $d_\ell a_{n+1+s\ell} = \left(d_\ell a_{s\ell}\right) a_{n+1}$, so it makes sense to also consider the one parameter group $d_\ell a_{s\ell}$. It is interesting to note that the above described strategy of Alice is in fact winning even if $s$ is replaced by any integer larger than the one described in \eqref{eq:s}, and that $d_\ell a_{s\ell}$ becomes closer in direction to the direction of $a_{n+1}$ as $s$ becomes larger.
\end{rem}
\appendix

{\zblue

\section{Proof of Proposition~\ref{prop:modifiedHAW}}

As we mentioned earlier the proof essentially follows the argument of \cite[Proposition~4.5]{FSU4}. First of all, note that if $1/3\le \beta<1$ then, with reference to Definition~\ref{def:hyperplaneWinning}, Alice can win by default on her first move by taking $A_1$ to be the closed $\beta$-neighborhood of any hyperplane passing through the centre of $B_0$. Thus, without loss of generality we can assume throughout this proof that Bob always chooses $\beta<1/3$ when he plays the restricted hyperplane absolute game. With this additional assumption the necessity part of Proposition~\ref{prop:modifiedHAW} becomes obvious. Indeed, to win the $\modified$ hyperplane absolute game Alice simply has to follow her strategy for the hyperplane absolute game and set $\eps$ to be exactly $\beta$ on each of her moves. By increasing $\eps$ to its largest possible value Alice will only limit the possible choices for Bob's next moves.
Additionally, the modified game has a greater restriction on Bob's moves, since the radii of his balls always satisfy $r_{n+1}=\beta r_n$. Note that since $\beta<1/3$ the game does not stop at a finite step. Therefore, the outcome of the $\modified$ hyperplane absolute game will lie in $S$ and Alice will win.

To prove the sufficiency requires some work. Suppose that $S$ is $\modified$ HAW, which means that Alice has a strategy to win the $\modified$ hyperplane absolute game. Following \cite{FSU4} and indeed Schmidt \cite{schmidt1}, by a strategy one can understand a sequence of maps $\cF_{n+1}$, $n=0,1,2,\dots$ which assign a legal move $A_{n+1}$ for Alice depending on Bob's previous moves $\beta,B_0,\dots,B_n$. This strategy is winning if Alice can win when she uses it. As was demonstrated by Schmidt \cite[Theorem~7]{schmidt1}, Alice always has a positional winning strategy for every $\modified$ hyperplane absolute winning set $S$. This means that for every $\beta\in\left(0,\frac13\right)$ there exists a map $\cF_\beta$ from the set of balls in $\R^d$ into the set of Alice's legal moves for the $\modified$ hyperplane absolute game such that $\cF_{n+1}\left(\beta,B_0,\dots,B_n\right)=\cF_\beta(B_n)$ is Alice's winning strategy. That is Alice can make her move only using the knowledge of $\beta$ and Bob's previous move $B_n$. From now on we fix a positional winning strategy for the restricted hyperplane absolute game, which we will use to define a winning strategy for the hyperplane absolute game.
Also, note that since $S$ is $\modified$ HAW, it has to be dense in $\R^d$; otherwise Bob can win the $\modified$ hyperplane absolute game by taking $B_0$ to contain no points of $S$.

Given any $0<\beta<1/3$, define the following map on balls $B(\bx,r)$ in $\R^d$ into hyperplane neighborhoods as follows: first find the unique integer $m_r\in\Z$ satisfying
\begin{equation}\label{vb+02}
(\beta/2)^{2m_r+1}\le r< (\beta/2)^{2m_r-1}\,;
\end{equation}
then write
\begin{equation}\label{vb+03}
\cF_{(\beta/2)^2}\left(B\left(\bx,(\beta/2)^{2m_r}\right)\right)=B\left(H,(\beta/2)^{2m_r+2}\right)
\end{equation}
where $H$ is a hyperplane in $\R^d$; and finally define
\begin{equation}\label{vb+01}
\cG_{\beta}(B(\bx,r)):=B\left(H,2(\beta/2)^{2m_r+2}\right).
\end{equation}
By the left hand side of \eqref{vb+02}, we have that
$$
2(\beta/2)^{2m_r+2}\le \beta r
$$
and therefore $\cG_{\beta}(B(\bx,r))$ represents a legal move of Alice for the hyperplane absolute game.
We claim that the map $\cG_{\beta}$ gives a positional winning strategy for Alice in the hyperplane absolute game.

Indeed, suppose that Bob chooses any $0<\beta<1/3$ on his first move and suppose that
$B_n=B(x_n,r_n)$ ($n=0,1,2,\dots$) and $A_1,A_2,\dots$ are the moves taken by Bob and Alice respectively in the hyperplane absolute game such that
\begin{equation}\label{vb+05}
A_{n+1}=\cG_{\beta}(B(\bx_n,r_n))=B\left(H_n,2(\beta/2)^{2m_{r_n}+2}\right)
\end{equation}
for all $n\ge0$, where $H_n$ is a hyperplane in $\R^d$ and $m_{r_n}$ arises from \eqref{vb+02} when $r=r_n$. Without loss of generality we can assume that $r_n\to0$ as $n\to\infty$ as otherwise Alice wins because $S$ is dense.

Now we extract a subsequence of $B_n$, say $B_{n_k}$ such that the radii $r_{n_k}$ are comparable to $(\beta/2)^{2k}$. To this end, first define $k_0\in\Z$ such that $(\beta/2)^{2k_0}<r_1$. Next, for each $k\ge k_0$ define $n_k$ such that
\begin{equation}\label{vb+04}
(\beta/2)^{ 2k+1}=\beta\,\tfrac12(\beta/2)^{2k} \le r_{n_k} < \tfrac12 (\beta/2)^{2k}\,.
\end{equation}
The existence of $n_k$ follows from the properties of the hyperplane absolute game, namely, the fact that $r_{n+1}/r_n\ge \beta $ for all $n\ge0$. Furthermore, observe that $n_k$ is increasing.
Also, by \eqref{vb+02},
\begin{equation}\label{vb+06}
m_{r_{n_k}}=k\,.
\end{equation}
Now define
$$
\tilde B_k=B\left(\bx_{n_k},(\beta/2)^{2k}\right)\qquad\text{and}\qquad\tilde A_{k+1}=\cF_{(\beta/2)^2}\left(\tilde B_k\right)=B\left(\tilde H_{k}, (\beta/2)^{2k+2}\right)
$$
for $k\ge k_0$, where $\tilde H_{k}$ are hyperplanes.

\noindent\textbf{Claim:}
$\tilde B_{k_0}$, $\tilde A_{k_0+1}$, $\tilde B_{k_0+1}$, $\tilde A_{k_0+2}$, \dots
is a sequence of legal moves taken by Bob and Alice in the $\modified$ hyperplane absolute game when Bob chooses $(\beta/2)^2$ as the parameter of the game.

The legality of Alice's moves $\tilde A_{k+1}$ is obvious from Definition~\ref{def:hyperplaneWinning}. To determine the legality of Bob's moves we have to verify that
\begin{equation}\label{vb+07}
\tilde B_{k+1}\subq \tilde B_{k}\setminus \tilde A_{k+1}
\end{equation}
for all $k\ge k_0$. First note that, since $B_n=B(x_n,r_n)$ for all $n=0,1,2,\dots$, we have that $\bx_{n_{k+1}}\in B_{n_{k+1}}$. Further, since $n_{k+1}\ge n_k+1$ and $B_n$ and $A_{n+1}$ for $n=0,1,2,\dots$ are the legal moves of Bob and Alice for the hyperplane absolute game, we have that
$$
B_{n_{k+1}}\subq B_{n_k+1}\subq B_{n_k}\setminus A_{n_{k}+1}\,.
$$
Therefore,
$$
\bx_{n_{k+1}}\in B_{n_k}\setminus A_{n_{k}+1}\,.
$$
Using the definitions of $B_{n_k}$, \eqref{vb+05}, \eqref{vb+04} and \eqref{vb+06}, the above implies that
\begin{equation}\label{vb+08}
\dist\left(\bx_{n_{k+1}},\bx_{n_{k}}\right)\le r_{n_k}<\tfrac12(\beta/2)^{2k}
\end{equation}
and
$$
\dist\left(\bx_{n_{k+1}},H_{n_{k}}\right)> 2(\beta/2)^{2m_{r_{n_k}}+2}=2(\beta/2)^{2k+2}\,.
$$
where $\dist(\,\cdot\,)$ is the Euclidean distance on $\R^d$.
Observe using \eqref{vb+03}, \eqref{vb+01} and \eqref{vb+05} that $H_{n_k}=\tilde H_k$. Then,
\begin{equation}\label{vb+09}
\dist\left(\bx_{n_{k+1}},\tilde H_{k}\right)> 2(\beta/2)^{2k+2}\,.
\end{equation}
Using the triangle inequality together with \eqref{vb+08} and \eqref{vb+09} gives
$$
B\left(\bx_{n_{k+1}},(\beta/2)^{2k+2}\right)\subq B\left(\bx_{n_{k}},(\beta/2)^{2k}\right)\setminus B\left(\tilde H_{k},(\beta/2)^{2k+2}\right),
$$
which is precisely \eqref{vb+07}. This completes the proof of the above claim.

Finally, observe that $\cap_nB_n$ is the same as $\cap_k\tilde B_k$ which has to be in $S$ by the above claim and the fact that Alice plays according to her winning strategy. Therefore, Alice wins the hyperplane absolute game, meaning that $S$ is HAW. The proof of Proposition~\ref{prop:modifiedHAW} is thus complete.
}

\begin{rem}
At no point in the above proof have we used the fact that $H$ and $H_n$ are hyperplanes. These could have been sets from any collection. Of course, on replacing hyperplanes with another collection of sets we would alter the hyperplane absolute game and the restricted hyperplane absolute game and the corresponding notions of winning. However, the fact that the nature of
$H$ and $H_n$ is irrelevant to the above proof means that the altered notions of winning sets and restricted winning sets are the same. In particular, this comment applies to the $k$-dimensional absolute game, which is a more general version of the hyperplane absolute game, studied in \cite{BFKRW}.
\end{rem}

\vspace{8 mm}
\emph{Acknowledgements}: EN would like to thank Elon Lindenstrauss and Barak Weiss for
indispensable discussions about $\badw$ and winning during the last decade, and to David Simmons for telling him about Proposition~\ref{prop:known}. EN acknowledges support from ERC 2020 grant
HomDyn (grant no. 833423) and ERC 2020 grant HD-App (grant no. 754475). EN would like to thank Anna Nesharim for joining him to the valleys and peaks. LY is supported in part by NSFC grant 11801384 and the Fundamental Research Funds for the Central Universities YJ201769. The authors are also very grateful to the anonymous reviewer for very helpful comments.

\bibliographystyle{alpha}
\bibliography{bibliography}

\end{document}